\documentclass[11pt,a4paper]{article}
 \usepackage{indentfirst, latexsym,bm}
 \usepackage{amsmath}
 \usepackage{pifont}
 \usepackage{amsfonts}
 \usepackage{mathrsfs}
 \usepackage{array}
 \usepackage{multirow} 
 \usepackage{graphicx}
 \usepackage{subfigure}
 \usepackage{picinpar}
 \usepackage{setspace}
 \RequirePackage[colorlinks,citecolor=blue,urlcolor=blue,linkcolor=blue]{hyperref}
 \usepackage[top=2cm,bottom=2cm, outer=2cm, inner=2cm]{geometry}
 
 \usepackage{caption}
 \usepackage[labelfont={bf,small},textfont={small}]{caption}
 
 \RequirePackage[numbers]{natbib}
 \usepackage{enumitem} 
 
 \usepackage{booktabs}
 \usepackage{authblk}
 
 \usepackage{threeparttable}
 \usepackage{mathrsfs,amsfonts,amsmath}
 
 \usepackage{color}
 
 \makeatletter
 \def\namedlabel#1#2{\begingroup
 	#2%
 	\def\@currentlabel{#2}%
 	\phantomsection\label{#1}\endgroup
 }
 \makeatother

 \numberwithin{figure}{section}

 \newcommand\email[1]{\href{mailto:#1}{ \nolinkurl{#1}}}

 \newtheorem{theorem}{Theorem}[section]
 \newtheorem{definition}[theorem]{Definition}
 \newtheorem{lemma}[theorem]{Lemma}
 \newtheorem{corollary}[theorem]{Corollary}
 \newtheorem{proposition}[theorem]{Proposition}
 \newtheorem{remark}[theorem]{Remark}
 \newtheorem{condition}[theorem]{Condition}
\newtheorem{example}[theorem]{Example}

 \def\blemma{\begin{lemma}}\def\elemma{\end{lemma}}
 \def\bproposition{\begin{proposition}}\def\eproposition{\end{proposition}}
 \def\ttheorem{\begin{theorem}}\def\etheorem{\end{theorem}}
 \def\bcorollary{\begin{corollary}}\def\ecorollary{\end{corollary}}
 \def\bremark{\begin{remark}}\def\eremark{\end{remark}}
 \def\bcondition{\begin{condition}}\def\econdition{\end{condition}}

 \def\benumerate{\begin{enumerate}}\def\eenumerate{\end{enumerate}}
 \def\bitemize{\begin{itemize}}\def\eitemize{\end{itemize}}

 \def\beqlb{\begin{eqnarray}}\def\eeqlb{\end{eqnarray}}
 \def\beqnn{\begin{eqnarray*}}\def\eeqnn{\end{eqnarray*}}
 \def\ar{\!\!\!&}

 \def\proof{\noindent{\it Proof.~~}}\def\qed{\hfill$\Box$\medskip}

\begin{document} 
 
 \title{\bf  Uniform Local Asymptotics for L\'evy Processes with Subexponential Jumps}
	
 \author{Hao Wu\footnote{School of Mathematical Sciences, Nankai University, China; email: wuhao.math@outlook.com} 
 \quad and\quad 
 Wei Xu\footnote{School of Mathematics and Statistics, Beijing Institute of Technology,  China; email: xuwei.math@gmail.com. Xu gratefully acknowledges financial support from the Fundamental Research Funds for the Central Universities, the National Natural Science Foundation of China (No. 11531001) and the National Key R\&D Program of China (No. 2023YFA1010103).} 
 }  

 \date{}
	
 \maketitle
 \begin{abstract}
 This paper is devoted to unifying the uniform local large-deviation asymptotics for a centered L\'evy process $X$ with subexponential jumps. 
 Our results assert that for any $\theta,\delta_0 >0$ and $K\geq 0$,
 \begin{equation*}
 \lim_{t\to\infty}\sup_{x\geq \theta t}  \sup_{|y|\leq Kb(x)} \sup_{\delta\in [\delta_0,\infty]} 	\sup_{0<s\leq t} \bigg| \frac{\mathbf{P}\big( X_s\in  (x-y,x-y+\delta] \big)}{ s\cdot \mathbf{P}\big( X_1\in  (x,x+\delta] \big) } -1 \bigg|  =0,
 \end{equation*}
 where the natural-scale function $b$ satisfies a polynomial growth condition. This provides a continuous-time and simultaneously uniform analogue of the results of Denisov et al. [Ann. Probab., 2008], while being established under a weaker moment assumption.
 
 \bigskip
 
 \noindent {\it MSC 2020 subject classifications:} 60G51, 60F10.
 
 \smallskip
 
 \noindent  {\it Keywords and phrases:} Large deviation,  L\'evy process, local probability, subexponentiality.

 \end{abstract}

  \section{Introduction} 
 \label{Sec.Introduction}
 \setcounter{equation}{0}
 
 Heavy-tailed large deviations have been extensively studied and also widely used in various fields over the past decades. 
 Unlike the classical Cram\'er regime, where large deviations are typically realized through an accumulation of many moderate fluctuations, heavy-tailed large deviations are governed by a fundamentally different mechanism, commonly referred to as the \textsl{big-jump principle}. 
 For a random walk $S$ started from $0$, this principle asserts that a large deviation is typically produced by one exceptionally large step, while the remaining steps fluctuate on a substantially smaller scale. At the level of tail probabilities, this phenomenon leads, under subexponentiality, to asymptotic relations of the form
  \beqnn
 \mathbf P\big(S_n>x\big)\sim n\cdot \mathbf{P}\big(S_1>x\big),\quad \mbox{as }x \to\infty.
 \eeqnn
  A substantially finer question is whether an analogous relation remains valid locally, namely whether  
  \beqlb\label{eqn.001}
  \mathbf P\big(S_n\in x+\varDelta_\delta\big)
  \sim n\cdot \mathbf P\big(S_1\in x+\varDelta_\delta\big) 
  \quad\mbox{with} \quad x+\varDelta_\delta=(x,x+\delta],
  \eeqlb
  uniformly throughout an appropriate large-deviation region. 
  Local asymptotics of this type are considerably more sensitive to spatial perturbations than their tail counterparts and therefore require, in addition to subexponentiality, suitable local regularity or insensitivity assumptions.

  A systematic theory of local subexponential distributions was developed by Asmussen et al. \cite{AsmussenFossKorshunov2003}. 
   They showed that much of the ordinary subexponential theory extends to this local setting and obtained applications to random walks, compound Poisson distributions, infinitely divisible distributions and renewal theory. 
  For random walks, a particularly general description of the big-jump domain was obtained by Denisov et al. \cite{DenisovDiekerShneer2008}. 
  They identified general conditions under which the approximation \eqref{eqn.001}  holds uniformly in the linear region $x\ge \theta n$ for any fixed $\delta \in(0,\infty]$, provided an appropriate moment condition and an insensitivity condition on the scale of the typical fluctuations of the remaining increments are satisfied.  
  Later, Denisov et al. \cite{DenisovFossKorshunov2010} investigated the one-big-jump principle for random walks stopped at an independent random time $\tau$. 
  Under appropriate assumptions, they proved that 
 \beqnn
  \mathbf P(S_\tau>x) \sim \mathbf E[\tau]\cdot\mathbf P\big(S_1> x\big),
  \eeqnn
  and also obtained related asymptotics for the maximum.  
  Furthermore, Denisov et al. \cite{DenisovVatutinWachtel2014}  studied local probabilities for a random walk with negative drift $\beta<0$ conditioned to stay positive. 
  Under the assumptions similarly as in \cite{DenisovDiekerShneer2008}, they proved that 
  \beqnn
  \mathbf P\big(S_n\in x+\varDelta_\delta, \, \tau_0>n\big) \sim \mathbf{E}\big[\tau_0\big]\cdot \mathbf{P} \big( S_1\in x-\beta n +\varDelta_\delta\big)
  \quad \mbox{with}\quad
  \tau_0= \inf\{n\geq1:S_n< 0\}.
  \eeqnn
  
The corresponding question for L\'evy processes is closely related to, but distinct from, the classical theory of subexponential infinitely divisible distributions (IDDs). 
Beginning with Embrechts et al. \cite{EmbrechtsGoldieVeraverbeke1979},  a substantial literature has established conditions under which the tail of an IDD is asymptotically equivalent to the tail of its L\'evy measure. 
Applied to a L\'evy process with L\'evy measure $\nu(dy)$, results of this type naturally lead under suitable assumptions to fixed-time relations of the form
\beqnn
\mathbf P\big(X_t\in x+\varDelta_\delta\big)
\sim t\cdot \nu\big(x+\varDelta_\delta\big),
\qquad x\to\infty.
 \eeqnn
This was improved in \cite{DoneyJones2012} at the level of tail probabilities to a uniform asymptotic for $x \geq \theta t$ as $t \to \infty$.
 Heavy-tailed asymptotics for L\'evy processes have also been studied extensively for path-dependent quantities. 
 In the subexponential setting, precise asymptotics are available for suprema of negatively driven L\'evy processes, including results uniform in the time horizon; see e.g. \cite{Korshunov2018}.  
 Passage times, overshoots and pre-passage positions under subexponential assumptions have been investigated in \cite{DoneyKluppelbergMaller2016}. A different line of work develops sample-path large-deviation principles for L\'evy processes with regularly varying or other heavy-tailed L\'evy measures including extrema, passage events and rare trajectories; see \cite{RheeBlanchetZwart2019} and reference therein.   
 
 This paper aims to establish a uniform local large-deviation principle for one-dimensional centered L\'evy processes, providing a continuous-time analogue of Corollary~2.1 in \cite{DenisovDiekerShneer2008}.  
 We should emphasize that a continuous-time analogue of that result has already been proposed by Xu \cite{Xu2021a} in the regularly varying setting; however, the proof given there is not rigorous. 
 
 \subsection{Main results}
 
 Consider a one-dimensional L\'evy process $X = \{X_t: t\ge 0\}$ defined on a probability space $(\Omega,\mathscr{F},\mathbf{P})$, with zero mean, 
 diffusion coefficient $\sigma\geq 0$ and \textsl{L\'evy measure} $\nu(dy)$ satisfying 
 \beqnn
 \int_{\mathbb{R}}\big(|y|\wedge |y|^2\big)\,\nu(dy)<\infty.
 \eeqnn

 To formulate the natural conditions on the L\'evy measure, we first introduce four classes of non-negative functions $f$ on $\mathbb{R}$:
 \begin{enumerate}
 	\item[$\bullet$] $\mathcal{AD}$: the class of \textsl{almost decreasing} functions, i.e., $	f(x) \geq c\cdot f(y)$ for some $c>0$ and all large $y\geq x>0$. 
 	
 	\item[$\bullet$] $\mathcal{S}d$: the class of \textsl{subexponential} functions, i.e., for all $y\in\mathbb{R}$, as $x\to\infty$,
 	\beqnn
 	\frac{1}{f(x)} \int_0^xf(x-z)f(z)dz \sim 2 \int_0^\infty f(z)dz<\infty
 	\quad \mbox{and}\quad 
 	f(x+y)\sim f(x).
 	\eeqnn
 	
 	\item[$\bullet$] $\mathcal{OR}$:  the class of \textsl{O-regularly varying} functions, i.e., for all $\lambda \geq 1$, 
 	\beqnn
 	0<\liminf _{x \to \infty} \frac{f(\lambda x)}{f(x)} \leq \limsup _{x \to \infty} \frac{f(\lambda x)}{f(x)}<\infty . 
 	\eeqnn 
 	
 	\item[$\bullet$] $\mathrm{RV}^\infty_{\beta}$: the class of regularly varying functions at infinity with index $\beta \in \mathbb{R}$, i.e., for all $\lambda >0$, 
 	\beqnn
 	\lim_{x \to \infty} \frac{f(\lambda x)}{f(x)} = \lambda^\beta. 
 	\eeqnn
 \end{enumerate}
 
  \begin{condition}\label{MainCondition01}
  Recall $\varDelta_\delta: =(0,\delta]$ for  $\delta\in(0,\infty]$.  The following hold for some constant $\alpha\in(1,2]$.
  	\begin{enumerate}
  		\item[(1)] {\rm(One-sided moment)} Assume that $\mathbf{E}\big[(X_1^+)^\alpha\big]<\infty$ with $X_1^+=X_1\vee 0$, which is equivalent to  
  		\beqnn
  		\int_1^\infty x^\alpha\,\nu(dx)<\infty . 
  		\eeqnn
  		
  		\item[(2)] {\rm(Fluctuation scale)} There exists an eventually non-decreasing positive function $b\in \mathcal{OR}$  such that 
  		\beqlb\label{eqn.Con01}
  		\liminf_{t\to\infty} \frac{b(t)}{t^{1/\alpha}} >0,\quad 
  		\limsup_{t\to\infty} \frac{b(t)}{t} =0
  		\quad \mbox{and}\quad 
  		\limsup_{K\to\infty} \limsup_{t\to\infty} \mathbf{P}\big( |X_t|\geq Kb(t) \big) =0.
  		\eeqlb
  		
  		\item[(3)] {\rm(Local heavy-tail)} For any $\delta >0$, assume that $\nu\big(x+\varDelta_\delta\big) \in \mathcal{AD}$, $x^{\alpha}\cdot \nu\big(x+\varDelta_\delta\big) \in \mathcal{S}d\cup\mathcal{OR}$ and 
  		\beqlb\label{eqn.Con02}
  		\limsup_{x\to\infty} \sup_{|y|\leq b(x)}  \bigg| \frac{ \nu\big( x-y+\varDelta_\delta\big) }{\nu\big( x+\varDelta_\delta\big)}-1\bigg| =0. 
  		\eeqlb 
  		\end{enumerate}
 \end{condition}
 
 The one-sided moment condition on the positive part is natural in our study of large positive deviations of $X$.  
 The almost-decreasing property of $\nu(x+\varDelta_\delta)$ rules out excessively large oscillations of the local L\'evy mass in the positive tail. 
 The heavy-tail regularity condition $\nu\big(x+\varDelta_\delta\big) \in \mathcal{AD}$ and $x^{\alpha}\cdot \nu\big(x+\varDelta_\delta\big) \in \mathcal{S}d\cup\mathcal{OR}$  appears frequently in local large-deviation theory, as it captures the characteristic convolution behavior underlying the one-big-jump mechanism. 
 The limit \eqref{eqn.Con02}, usually referred to as the \textsl{insensitivity condition}, states that the local L\'evy mass near the level $x$ is asymptotically unaffected by perturbations whose magnitude is at most  $b(x)$. 
 A direct probabilistic interpretation is that, in a one-big-jump realization of the event $\{X_t\in x+\varDelta_\delta\}$, the exceptional positive jump produces the displacement of order $x$, whereas the remaining part of the process contributes a fluctuation of order $b(x)$. 
 
 We emphasize that Condition~\ref{MainCondition01}(2) is rather mild and can always be satisfied. 
 Indeed, the law of large numbers induces that $X_t/t\overset{\rm p}\to 0$ as $t\to\infty$. Hence, we may choose an increasing sequence $\{T_n\}_{n\geq0}$ with $T_0=0$ and $T_1=1$ such that $T_{n}/T_{n-1}\geq n$ and $\mathbf{P}\big(|X_t|/t \geq 1/n^2\big) \leq 1/n$ for all $n\geq 1$ and $t\geq T_n$.  
 Consider a continuous piecewise linear function $\ell $ on $\mathbb{R}_+$ defined by 
 	\beqnn
 	\ell(t)=  n+ \frac{t-T_n}{T_{n+1}-T_n} \in[n,n+1], \qquad \mbox{for }t\in[T_n,T_{n+1}).
 	\eeqnn
 It is straightforward to verify that $\ell(t)\to\infty$ and $\ell \in \mathrm{RV}^\infty_0$. 
 Now set 
 \beqlb\label{eqn.002}
 b(0)=1\quad \mbox{and}
 \quad 
 b(t)=t/\ell(t),\quad t>0.
 \eeqlb 
 For  $t\in[T_n,T_{n+1})$, we have $\ell(t)- t\cdot \ell'(t) \geq 0$ and hence  $b'(t)\geq 0$. 
 Therefore, $b$ is eventually non-decreasing and the first two limits in \eqref{eqn.Con01} hold. 
 It remains to verify the last limit in \eqref{eqn.Con01}. 
 Fix $\varepsilon\in(0,1)$.
 For any $n\geq 2/\varepsilon$ and $t\in[T_n,T_{n+1})$, we have $n\leq \ell(t)\leq n+1$  and then 
 	\beqnn
 	\mathbf{P}\big( |X_t|/b(t)/\geq \varepsilon \big)
 	\ar=\ar   \mathbf{P}\big( |X_t|/t\geq \varepsilon/\ell(t) \big) 
 	 \leq   \mathbf{P}\big( |X_t|/t\geq 1/n^2  \big) \leq 1/n,
 	\eeqnn
 	which goes to $0$ as $t\to\infty$ and then  \eqref{eqn.Con01} follows. 
 Hence, the substantive restriction on the choice of $b$ comes from the insensitivity condition \eqref{eqn.Con02}. The next example illustrates this point. 
 
 \begin{example}
 Suppose that $X$ has L\'evy measure
 $$\nu(dx)= \mathbf 1_{\{x\geq1\}} \cdot\big( \mathbf{1}_{\{x-[x] \in [0,1/2)\}} +2\cdot \mathbf{1}_{\{x-[x] \in [1/2,1)\}}  \big)\cdot
 x^{-1-\alpha-\varepsilon} \,dx$$
 for some $\varepsilon>0$ and $b$ is given by \eqref{eqn.002}.
 Note that $ \nu\big(n+\varDelta_{1/4}\big) \sim n^{-1-\alpha-\varepsilon}/4$ and $\nu\big(
 n-1/2+\varDelta_{1/4} \big) \sim n^{-1-\alpha-\varepsilon}/2$ as $n\to\infty$,  we have 
 \beqnn
 \limsup_{n\to\infty}\sup_{|y|\leq b(n)}
 \bigg| \frac{ \nu\big(n-y+\varDelta_{1/4}\big) }{ \nu\big(n+\varDelta_{1/4}\big) } -1 \bigg|
 \geq
 \limsup_{n\to\infty} \bigg|  \frac{ \nu\big(n-\frac12+\varDelta_{1/4}\big) }{ \nu\big(n+\varDelta_{1/4}\big) }-1 \bigg| =1. 
 \eeqnn 
 Hence, the insensitivity condition \eqref{eqn.Con02} fails for this example. 
 \end{example}

 Assumptions in Condition~\ref{MainCondition01} are closely related to, but differ in several respects from, those appearing in Corollary~2.1 in \cite{DenisovDiekerShneer2008}. 
 First, our one-sided moment condition is weaker than the two-sided moment assumption $\mathbf{E}\big[|S_1|^\alpha\big]<\infty$ employed in \cite{DenisovDiekerShneer2008}, which  allows $n^{1/\alpha}$  to serve as the natural scale sequence for random walks. 
 In our L\'evy setting, we decouple this fluctuation requirement from the positive-tail moment condition: the negative tail of $X$ is permitted to generate fluctuations on a scale larger than $t^{1/\alpha}$, provided these fluctuations remain controllable by a sublinear scale $b(t)$, while  the finite positive $\alpha$-moment--which underpins the one-big-jump mechanism on the positive side--is retained. 
 Second, their assumption that $\mathbf P(S_1\in x+\varDelta_\delta)$ is insensitive to shifts of order $x^{1/\alpha}$ is replaced here by insensitivity of the local L\'evy mass $\nu(x+\varDelta_\delta)$ to perturbations of order $b(x)$. This allows the fluctuation scale to be adapted to the L\'evy process rather than fixed a priori at  $x^{1/\alpha}$. 
 Finally, while the random-walk result is formulated for a fixed interval $\varDelta_\delta$, our assumptions are imposed for every $\delta>0$ and yield an approximation uniform over all $\delta\geq \delta_0$. 
 The truncation and small-steps sequences entering the general theorem in \cite{DenisovDiekerShneer2008} do not appear explicitly in our statement; their roles are instead taken over by the L\'evy--It\^o decomposition and the corresponding estimates for the compound-Poisson large-jump component and the remaining small-fluctuation component. 
  	
 The relation between Conditions~(2) and (3) is fundamental. 
 Condition~(2) asserts that the contribution of the ordinary part of the L\'evy process is typically of order at most $b(t)$, whereas Condition~(3) asserts that the local L\'evy measure is insensitive to perturbations on precisely this scale. Since $b(t)=o(t)$, these fluctuations are much smaller than the large-deviation levels $x\geq \theta t$.  
 Our main theorem therefore formalizes the following one-big-jump mechanism: a local large deviation of $X_s$ to a level of order $t$ is generated by one exceptionally large positive jump, while all remaining fluctuations occur on the smaller $b$-scale and have no first-order effect on the required size of that jump. 
  	
  	 \begin{theorem}\label{Main.Thm} 
  		Under Condition~\ref{MainCondition01},	for any constants $\theta,\delta_0>0 $ and $K\geq 0$, we have as $t\to\infty$, 
  		\beqlb\label{main.eqn}
  	\sup_{x\geq \theta t}  \sup_{|y|\leq Kb(x)} \sup_{\delta\in [\delta_0,\infty]}	\sup_{0<s\leq t} \bigg| \frac{\mathbf{P}\big( X_s\in  x-y+\varDelta_\delta \big)}{ s\cdot  \nu\big(x+\varDelta_\delta\big) } -1 \bigg| \to 0. 
  		\eeqlb
  	\end{theorem}
  	
 In particular, the limit \eqref{main.eqn} implies that $\mathbf{P}\big( X_1\in  x+\varDelta_\delta \big) \sim  \nu\big(x+\varDelta_\delta\big)$ uniformly in $\delta\in [\delta_0,\infty]$ as $x\to\infty$. 
 Consequently, the approximation
  	\beqnn
  	\mathbf{P}\big( X_s\in  x-y+\varDelta_\delta \big) 
  	\sim  s\cdot  \mathbf{P}\big( X_1\in  x-y+\varDelta_\delta \big)
  	\eeqnn 
 holds simultaneously with respect to the observation time $s\in(0,t]$, the large-deviation level $x\ge\theta t$, the spatial perturbation $|y|\le Kb(x)$, and the interval length $\delta\ge\delta_0$. 
 This uniform simultaneity is stronger than the corresponding result for random walks in Denisov et al.  \cite{DenisovDiekerShneer2008}, where the analogous approximation \eqref{eqn.001} is established only uniformly in the spatial variable $x\geq \theta n$.
 Accordingly, our theorem should be viewed as a continuous-time and fully uniform counterpart of the local big-jump principle in \cite{DenisovDiekerShneer2008}, rather than as a mere substitution of $n$ by $t$. 
 
 When specialized to the most important subclasses of heavy-tailed L'evy measures considered widely in the literature, the following result shows that our uniform local large-deviation asymptotics hold under minimal and readily verifiable conditions.

 \begin{corollary}\label{MainCorollary}
  Suppose that there exists a constant $\alpha>1$ such that $x\mapsto\nu\big(x+\varDelta_\delta\big) \in \mathrm{RV}^\infty_{-\alpha-1}$ for any $\delta>0$. 
  Then, for any constants $\lambda,\delta_0>0$ and any positive function $\vartheta$ on $\mathbb R+$ satisfying $\vartheta_x/x\to 0$ as $x\to\infty$, we have, as $t\to\infty$,
 	\beqlb\label{eqn.323}
 	\sup_{0<s\leq  t} \sup_{x\geq \lambda t} \sup_{|y|\leq \vartheta_x}  \sup_{\delta\in[\delta_0,\infty]} 
 	\bigg| \frac{\mathbf{P}\big(X_s   \in x -y+ \varDelta_\delta \big)}{s\cdot \nu\big(x + \varDelta_\delta \big)  } -1\bigg|    
 	\to 0. 
 	\eeqlb
 	
 	\end{corollary}

  	  \smallskip
  	{\it \textbf{Organization of this paper.}}  
  	In Section~\ref{Sec.AuxiliaryResult}, we collect several auxiliary results on the local asymptotics of the L\'evy measure. Section~\ref{Sec.Proofs} then presents the detailed proof of Theorem~\ref{Main.Thm}, with the proofs of two important auxiliary lemmas postponed to the two subsections that follow.

  	\medskip

  	\noindent \textbf{Notation.} 
  	For any $x\in\mathbb{R}$,  let $[x]$ be its integer part and $x^+ := x\vee 0=\max\{x,0\}$.
  	We make the convention that for any $s\leq t\in\mathbb{R}$,
  	\begin{equation*}
  		\int_{s}^{t}=\int_{(s,t]}\quad \mbox{and} \quad \int_{t}^{\infty}=\int_{(t,\infty)}.
  	\end{equation*}

 \section{Some auxiliary results} 
 \label{Sec.AuxiliaryResult}
 \setcounter{equation}{0}
 
 In this section, we provide several auxiliary results for the L\'evy measure that play a very important role in the proof of our main theorem. 
 Without loss of generality, in the sequel we may make the convention that 
 \begin{center}
 	{\bf (i) $\overline{\nu}(1)=1$;\quad  (ii) $b$ is non-decreasing on $\mathbb{R}_+$  and $b(0)=1$. } 
 \end{center}
 The next proposition follows directly from the non-decreasing and $O$-regularly varying property of $b$.  The detailed proof is omitted. 
 \begin{proposition}\label{Prop.204}
 	For any $\eta \geq 1$, there exist two constants
 	$C_\eta,x_0 \geq 1$ such that for all $x\geq x_0$, 
 	\beqnn
 	\sup_{0<y\leq \eta  x} \frac{b(y)}{b(x)} \leq C_\eta 
 	\quad \mbox{and}\quad 
 	\inf_{x/\eta\leq y\leq x} \frac{b(y)}{b(x)} \geq \frac{1}{C_\eta}. 
 	\eeqnn 
 \end{proposition}

 \begin{proposition}\label{Prop.2070}
 	For any $K\geq 0$ and $\delta>0$, we have  as $x\to\infty $,
 	\beqlb\label{eqn.30012}
 	 \sup_{|y|\leq Kb(x)} \bigg| \frac{ \nu\big(x-y+\varDelta_{\delta}\big) }{ \nu\big(x+\varDelta_{\delta}\big)} -1 \bigg| \to 0.
 	\eeqlb
 	\end{proposition}
 	\proof  Here we just to consider the case of $K>1$. 
 	 By Proposition~\ref{Prop.204}  and the eventually non-decreasing property of $b$, there exists a constant $C_0\ge1$ such that $b(x)\le C_0 b(u)$ for all sufficiently large $x$ and all $u\in[x/2,2x]$. 
 	 Choose a fixed integer $m\ge C_0K$. 
 	 Since $b(x)/x\to 0$ as $x\to\infty$,  we have $x-jy/m\in[x/2,2x]$ for all sufficiently large $x$, all $|y|\le Kb(x)$ and $0\le j<m$, which follows that 
 	\beqnn
 	\frac{|y|}{m}
 	\ar\leq\ar
 	\frac{Kb(x)}{m}
 	\le \frac{b(x)}{C_0}
 	\le b\big(x-jy/m\big). 
 	\eeqnn
  Our assumption \eqref{eqn.Con02} gives that as $x\to\infty$, uniformly in $|y|\le Kb(x)$,
 	\beqnn
 	\frac{\nu(x-y+\Delta_\delta)}{\nu(x+\Delta_\delta)}
 	\ar=\ar
 	\prod_{j=0}^{m-1}
 	\frac{\nu\big(x-(j+1)y/m+\Delta_\delta\big)}
 	{\nu\big(x-jy/m+\Delta_\delta\big)}
 	\to 1.
 	\eeqnn
 	This proves \eqref{eqn.30012} and the proof ends.
 	\qed

 \begin{proposition}\label{Prop.207}
 	For any $K\geq 0$ and $\delta_1, \delta_2>0$, we have as $x\to\infty $,
 	\beqlb\label{eqn.2011}
 	\sup_{|y|\leq Kb(x)} \bigg| \frac{ \nu\big(x-y+\varDelta_{\delta_1}\big) }{ \nu\big(x+\varDelta_{\delta_2}\big)} -\frac{\delta_1}{\delta_2} \bigg| \to 0.
 	\eeqlb
 \end{proposition}
 \proof  
  For any  integers $n\geq 1$ and $k\geq 1$,  Proposition~\ref{Prop.2070} induces that as $x\to\infty $, 
 \beqnn
 \sup_{|y|\leq Kb(x)} \bigg|\frac{\nu\big(x-y+\varDelta_{\frac{k}{2^n}}\big)}{\nu\big(x+\varDelta_{\frac{1}{2^n}}\big)} -k \bigg|
 \ar\leq\ar \sum_{j=0}^{k-1} 	\sup_{|y|\leq Kb(x)} \bigg|\frac{\nu\big(x -y+ \frac{j}{2^n}+\varDelta_{\frac{1}{2^n}}\big)}{\nu\big(x+\varDelta_{\frac{1}{2^n}}\big)} -1\bigg| \to 0.
 \eeqnn
 This immediately yields that for all integers $k,l \geq 1$,
 \beqlb\label{eqn.203}
 \lim_{x\to\infty} 	\sup_{|y|\leq Kb(x)} \bigg|\frac{\nu\big(x-y+\varDelta_{\frac{k}{2^n}}\big)}{\nu\big(x+\varDelta_{\frac{\ell}{2^n}}\big)} -\frac{k}{l} \bigg| =0.
 \eeqlb
 For any $\delta_1,\delta_2>0$ and any large integer $n\geq 1$, one can always find two integers $k_{n,1}, k_{n,2}\geq 1$ such  that $\frac{k_{n,1}}{2^n}<\delta_1\leq \frac{k_{n,1}+1}{2^n}$ and $\frac{k_{n,2}}{2^n}<\delta_2\leq \frac{k_{n,2}+1}{2^n}$. 
 The non-decreasing property of $\nu\big(x+\varDelta_{\delta}\big)$ in $\delta$ induces that 
 \beqnn
 \frac{\nu\big(x-y+\varDelta_{\frac{k_{n,1}}{2^n}}\big) }{\nu\big(x+\varDelta_{\frac{k_{n,2}+1}{2^n}}\big) }
 \leq   \frac{ \nu\big(x-y+\varDelta_{\delta_1}\big) }{ \nu\big(x+\varDelta_{\delta_2}\big)} 
 \leq   \frac{\nu\big(x-y+\varDelta_{\frac{k_{n,1}+1}{2^n}}\big) }{\nu\big(x+\varDelta_{\frac{k_{n,2}}{2^n}}\big)}
 \eeqnn
 The limit \eqref{eqn.203} tells that 
 \beqnn
 \liminf_{x\to\infty} \inf_{|y|\leq Kb(x)}\frac{\nu\big(x-y+\varDelta_{\frac{k_{n,1}}{2^n}}\big) }{\nu\big(x+\varDelta_{\frac{k_{n,2}+1}{2^n}}\big) } \geq \frac{ k_{n,1}}{ k_{n,2}+1}
 \quad\mbox{and}\quad
 \limsup_{x\to\infty} \sup_{|y|\leq Kb(x)} \frac{\nu\big(x-y+\varDelta_{\frac{k_{n,1}+1}{2^n}}\big) }{\nu\big(x+\varDelta_{\frac{k_{n,2}}{2^n}}\big)} \leq \frac{k_{n,1}+1}{k_{n,2}},
 \eeqnn
 which follows that for all large $n\geq 1$,
 \beqnn
 \frac{ k_{n,1}}{ k_{n,2}+1}\leq  \liminf_{x\to\infty} \inf_{|y|\leq Kb(x)}\frac{ \nu\big(x-y+\varDelta_{\delta_1}\big) }{ \nu\big(x+\varDelta_{\delta_2}\big)} 
 \leq  \limsup_{x\to\infty} \sup_{|y|\leq Kb(x)}\frac{ \nu\big(x-y+\varDelta_{\delta_1}\big) }{ \nu\big(x+\varDelta_{\delta_2}\big)}  
 \leq \frac{k_{n,1}+1}{k_{n,2}},
 \eeqnn
 and then \eqref{eqn.2011} holds, since both the lower and upper bounds converge to $\delta_1/\delta_2$ as $n\to\infty$. 
 \qed 
 
 \begin{proposition}\label{Prop.205}
 	For any $K\geq 0$ and $\delta_1\geq \delta_0>0$, we have  as $x\to\infty$,
 	\beqlb\label{eqn.202}
 	\sup_{|y|\leq Kb(x)}  \sup_{\delta\in [\delta_0,\delta_1]}   \bigg| \frac{ \nu\big(x-y+\varDelta_\delta\big) }{ \nu\big(x+\varDelta_\delta\big)}
 	-1 \bigg|  \to 0. 
 	\eeqlb
 	Moreover,  there exists a constant $x_0>0$ such that  
 	\beqlb\label{eqn.2021}
 	\sup_{x\geq x_0}\sup_{y\geq  x}  \sup_{\delta\in [\delta_0,\delta_1]} 
 	\frac{\nu\big(y+\varDelta_\delta\big)}{\nu\big(x+\varDelta_\delta\big)}<\infty. 
 	\eeqlb
 	
 \end{proposition}
 \proof We first prove the limit \eqref{eqn.202}, which follows from Proposition \ref{Prop.2070} when $\delta_1=\delta_0$. 
 We now prove it with $\delta_1>\delta_0$. 
 For any $\epsilon \in(0,1)$, we take a finite partition $\delta_0= r_0<r_1<\cdots r_{k_\epsilon}= \delta_1$ of the interval $[\delta_0,\delta_1]$ with mesh size at most $\epsilon$. 
 The increase of $\nu\big(x-y+\varDelta_\delta\big)$ in $\delta$ yields that  for any $j=0,\cdots, k_\epsilon-1$,
 \beqnn
 \nu\big(x-y+\varDelta_{r_j}\big)
 \leq 
 \inf_{ \delta\in [r_j,r_{j+1}]} \nu\big(x-y+\varDelta_\delta\big)  
 \leq \sup_{ \delta\in [r_j,r_{j+1}]} \nu\big(x-y+\varDelta_\delta\big)
 \leq
 \nu\big(x-y+\varDelta_{r_{j+1}}\big),
 \eeqnn
 which follows that 
 \beqnn
 \frac{\nu\big(x-y+\varDelta_{r_{j}}\big)}{\nu\big(x+\varDelta_{r_{j+1}}\big)}
 \leq 
 \inf_{ \delta\in [r_j,r_{j+1}]} \frac{\nu\big(x-y+\varDelta_\delta\big)}{\nu\big(x+\varDelta_\delta\big)}
 \leq 
 \sup_{ \delta\in [r_j,r_{j+1}]} \frac{\nu\big(x-y+\varDelta_\delta\big)}{\nu\big(x+\varDelta_\delta\big)}
 \leq 
 \frac{\nu\big(x-y+\varDelta_{r_{j+1}}\big)}{\nu\big(x+\varDelta_{r_{j}}\big)}
 \eeqnn
 and 
 \beqnn
 \inf_{j=0,\cdots, k_\epsilon-1}\frac{\nu\big(x-y+\varDelta_{r_{j}}\big)}{\nu\big(x+\varDelta_{r_{j+1}}\big)}
 \ar\leq\ar   \inf_{\delta \in[\delta_0,\delta_1]} \frac{\nu\big(x-y+\varDelta_\delta\big)}{\nu\big(x+\varDelta_\delta\big)} \cr
 \ar\leq\ar \sup_{\delta \in[\delta_0,\delta_1]} \frac{\nu\big(x-y+\varDelta_\delta\big)}{\nu\big(x+\varDelta_\delta\big)}
 \leq \sup_{j=0,\cdots, k_\epsilon-1} \frac{\nu\big(x-y+\varDelta_{r_{j+1}}\big)}{\nu\big(x+\varDelta_{r_{j}}\big)}. 
 \eeqnn
 By Proposition~\ref{Prop.207}, we have
 \beqnn
 \lim_{x\to\infty}  \inf_{j=0,\cdots, k_\epsilon-1} \inf_{|y|\leq Kb(x)} \frac{\nu\big(x-y+\varDelta_{r_{j}}\big)}{\nu\big(x+\varDelta_{r_{j+1}}\big)} 
 =\inf_{j=0,\cdots, k_\epsilon-1} \frac{r_j}{r_{j+1}}\geq \frac{\delta_0}{\delta_0+\epsilon}
 \eeqnn
 and 
 \beqnn
 \lim_{x\to\infty}  \sup_{j=0,\cdots, k_\epsilon-1} \sup_{|y|\leq Kb(x)}\frac{\nu\big(x-y+\varDelta_{r_{j+1}}\big)}{\nu\big(x+\varDelta_{r_{j}}\big)}
 = \sup_{j=0,\cdots, k_\epsilon-1} \frac{r_{j+1}}{r_j}\leq \frac{\delta_0+\epsilon}{\delta_0}, 
 \eeqnn
 which follows that 
 \beqnn
 \frac{\delta_0}{\delta_0+\epsilon}\leq \liminf_{x\to\infty}  \inf_{|y|\leq Kb(x)}\inf_{\delta \in[\delta_0,\delta_1]} \frac{\nu\big(x-y+\varDelta_\delta\big)}{\nu\big(x+\varDelta_\delta\big)}  
 \ar\leq\ar  \limsup_{x\to\infty}\sup_{|y|\leq Kb(x)} \sup_{\delta \in[\delta_0,\delta_1]} \frac{\nu\big(x-y+\varDelta_\delta\big)}{\nu\big(x+\varDelta_\delta\big)}\leq \frac{\delta_0+\epsilon}{\delta_0}
 \eeqnn
 and then the desired limit \eqref{eqn.202} follows as $\epsilon\to 0+$. 
 The uniform upper bound \eqref{eqn.2021} follows directly from  \eqref{eqn.202} as well as the non-decreasing property in $\delta$ and the almost decrease in $x$ of $\nu\big(x+\varDelta_\delta\big)$. 
 \qed

 \begin{proposition}\label{Prop.202}
 	For any $ \delta_0>0$, there exist two constants $C,x_0>0$ such that for all $x\geq x_0$,
 	\beqnn
 	\inf_{\delta\geq \delta_0} \nu\big(x+\varDelta_\delta\big) \geq \exp\big\{ -C\cdot  x^{1-1/\alpha} \big\}.
 	\eeqnn
 \end{proposition}
 \proof Since $\nu\big(x+\varDelta_\delta\big)$ is non-decreasing in $\delta$, it suffices to prove 
 \beqlb\label{eqn.201}
 \nu\big(x+\varDelta_{\delta_0}\big) \geq \exp\big\{ -C\cdot x^{1-1/\alpha} \big\},
 \eeqlb 
 for large $x$ and some $C>0$. 
 In view of the first limit in \eqref{eqn.Con01},  we have $x^{1/\alpha} \leq C_1b(x)$ uniformly in large $x$ and then  by Proportion~\ref{Prop.2070},
 \beqnn 
 \sup_{0\leq y\leq x^{1/\alpha}}
 \bigg| \frac{ \nu\big(x-y+\varDelta_{\delta_0}\big) }{ \nu\big(x+\varDelta_{\delta_0}\big)}-1\bigg|	\leq \sup_{|y|\leq C_1b(x)}
 \bigg| \frac{ \nu\big(x-y+\varDelta_{\delta_0}\big) }{ \nu\big(x+\varDelta_{\delta_0}\big) }-1\bigg|
 \to0.
 \eeqnn
 By Lemma~2.3 in~\cite{DenisovVatutinWachtel2014}, there exists a constant $C>0$ such that $\big|\log\nu\big(x+\varDelta_{\delta_0}\big)
 \big|\leq C\cdot x^{1-1/\alpha}$ uniformly in large $x$ and then the lower bound \eqref{eqn.201} follows.
 \qed

 \begin{proposition}\label{Prop.203}
 	For any $\eta>0$ and any $\delta_1\geq \delta_0>0 $, there exists a constant $x_0>0$ such that for all	$x\geq x_0$ and $y\in [0,x/2]$, 
 	\beqnn
 	\sup_{\delta\in [\delta_0,\delta_1]}  \frac{\nu\big(x-y+\varDelta_\delta\big)}{\nu\big(x+\varDelta_\delta\big)} \leq   \exp \Big\{ \eta \cdot \Big(\frac{y}{b(x)}+1\Big) \Big\}. 
 	\eeqnn 
 \end{proposition} 
 \proof
 By Proposition~\ref{Prop.204}, there exist constants $C_0,z_0>0$ such that $ b(x)\leq C_0b(u)$ for all $x\ge 2z_0$ and $u\in[x/2,x]$. 
 Since $b(u)=o(u)$ as $u\to\infty$, by Proposition~\ref{Prop.205}, we may increase $z_0$ so that $b(u)<u/2$ for all $u\ge z_0$ and
 \beqnn
 \sup_{u\ge z_0}\sup_{0\le v\le C_0 b(u)}
 \sup_{\delta\in[\delta_0,\delta_1]}
 \frac{\nu(u-v+\Delta_\delta)}{\nu(u+\Delta_\delta)}
 \ar\leq\ar e^\eta.
 \eeqnn
 For any $x\ge 2z_0$ and $y\in[0,b(x)]$, this gives the desired bound. Now let $y\in[b(x),x/2]$. For every integer $0\le k\le\lfloor y/b(x)\rfloor$, we have 
 \beqnn
 \frac{x}{2}\le x-kb(x)\le x
 \quad\mbox{and}\quad 
 b(x)\le C_0 b\big(x-kb(x)\big).
 \eeqnn
 Note that $y-\lfloor y/b(x)\rfloor b(x)<b(x)$, the preceding estimate applies to every factor below, uniformly in $\delta\in[\delta_0,\delta_1]$:
  \beqnn
 \frac{\nu\big(x-y+\varDelta_\delta\big)}{\nu\big(x+\varDelta_\delta\big)} 
 \ar=\ar \frac{\nu\big(x-[\frac{y}{b(x)}]\cdot b(x)-(y-[\frac{y}{b(x)}]\cdot b(x))+\varDelta_\delta\big)}{\nu\big(x-[\frac{y}{b(x)}]\cdot b(x)+\varDelta_\delta\big)}   \times \prod_{k=0}^{[\frac{y}{b(x)}]-1} \frac{\nu\big(x-  kb(x)-b(x)+\varDelta_\delta\big)}{\nu\big(x-kb(x)+\varDelta_\delta\big)}  ,
 \eeqnn
 which can be bounded by $\exp\{\eta \big( [\frac{y}{b(x)}] +1 \big)\} \leq \exp\{\eta\big( \frac{y}{b(x)} +1 \big)\}$ uniformly in $\delta\in [\delta_0,\delta_1]$. Then the desired upper bound holds with $x_0=2 z_0$.
 \qed

 \section{Proofs}
 \label{Sec.Proofs}
 \setcounter{equation}{0} 
 
 This section is devoted to providing the detailed proof of our main theorem. 
 In order to make it much easier to understand, we first formulate two important auxiliary lemmas and give their proofs later in  the next two subsections.

 By the L\'evy-It\^o  decomposition theorem,  the process $X$ admits the following representation: 
 \beqnn
 X_t:= \sigma \cdot B_t + \int_0^t \int_{\mathbb{R}} y \widetilde{N}(ds,dy),\quad t\geq 0,
 \eeqnn
 where $B$ is a standard Brownian motion and $\widetilde{N}(ds,dy)$ is a compensated Poisson random measure on $(0,\infty)\times \mathbb{R}$ with intensity $ds\, \nu(dy)$. 
 It can also be decomposed into the following two independent martingales
 \beqnn
 M_t := \sigma \cdot B_t + \int_0^t \int_{-\infty}^1 y \widetilde{N}(ds,dy)
 \quad \mbox{and}\quad 
 Z_t:=\int_0^t \int_1^{\infty}y \widetilde{N}(ds,dy),\quad t\geq 0,
 \eeqnn 
 which contain only jumps no larger than $1$ and only jumps larger than $1$, respectively.

 \begin{lemma} \label{Lemma.302}
 	The following assertions hold.
 	\begin{enumerate}
 		\item[(1)] The following limit holds
 		\beqnn
 		\lim_{K\to\infty}\limsup_{t\to\infty} \mathbf{P}\Big( \frac{|M_t|}{b(t)}\geq K  \Big)  =0. 
 		\eeqnn
 		
 		\item[(2)] There exists a constant $t_0>0$ such that  
 		\beqnn
 		\sup_{t\geq t_0}\mathbf{E}\bigg[ \exp\Big\{ \frac{M_t}{b(t)}\Big\} \bigg] <\infty.
 		\eeqnn
 		
 		\item[(3)] For any $\eta>0$, there exist two constants $c_\eta,C_\eta  >0$ such that for all $t\geq 0$ and $y\geq \eta t$, 
 		\beqnn
 		\mathbf{P}\big(M_t>y\big) \leq C_\eta\cdot \exp\big\{-c_\eta\cdot y\big\}. 
 		\eeqnn
 		
 		\item[(4)] 	For any $T>0$ and $\varepsilon\in(0,1)$, we have  as $y\to\infty$,
 		\beqnn
 		\exp\{\varepsilon y\log y\}\cdot \sup_{0<s\leq T} 
 		\frac{ \mathbf P\big(M_s>y \big) }{ s } \to 0.
 		\eeqnn
 	\end{enumerate}
 \end{lemma}

 \begin{lemma}\label{Lemma.301}
 	For any constants $\theta>0$, $ K\geq 0$ and
 	$\delta_1\geq \delta_0>0$, we have as $t\to\infty$, 
 	\beqlb\label{eqn.301}
 	\sup_{0<s\leq t} \sup_{x\geq \theta t} \sup_{|y|\leq Kb(x)} \sup_{\delta\in[\delta_0,\delta_1]} 
 	\bigg| 	\frac{ \mathbf P\big( Z_s\in x-y+\varDelta_\delta \big)}{s\cdot \nu\big(x+\varDelta_\delta\big)} -1 \bigg| \to 0.
 	\eeqlb
 	Moreover, there exist some constants $C,t_0>0$ such that for all $t\geq t_0$,
 	\beqlb\label{eqn.3011}
 	\sup_{0<s\leq t} \sup_{x\geq \theta t} \sup_{y\leq Kb(x)} \sup_{\delta\in[\delta_0,\delta_1]}
 	\frac{ \mathbf P\big( Z_s\in x-y+\varDelta_\delta \big)}{s\cdot \nu\big(x+\varDelta_\delta\big)}  \leq C.
 	\eeqlb
 \end{lemma}

 \textit{\textbf{Proof of Theorem~\ref{Main.Thm}.}} 
 Without loss of generality, we prove the theorem for the case $\delta_0=1$; the general case follows by an identical argument.
 The proof consists of two steps. First, we establish that the limit \eqref{main.eqn} holds uniformly in  $\delta\in [1,2]$, that is,
 \beqlb\label{main.eqn01}
 \sup_{0<s\leq t}\sup_{x\geq \theta t}  \sup_{|y|\leq Kb(x)}  \sup_{\delta\in[1,2]} \bigg| \frac{\mathbf{P}\big( X_s\in  x-y+\varDelta_\delta \big)}{ s\cdot  \nu\big(x+\varDelta_\delta\big) } -1 \bigg| \to 0 . 
 \eeqlb
 Second, we extend this uniform convergence to all $\delta \geq 1$ by the method of interval partition. 
 
 \textbf{Step 1.} First, by the  triangle inequality we  have 
 \beqnn
 \bigg| \frac{\mathbf{P}\big( X_s\in  x-y+\varDelta_\delta \big)}{ s\cdot  \nu\big(x+\varDelta_\delta\big) } -1 \bigg| 
 \leq \bigg| \frac{\nu\big(x-y+\varDelta_\delta\big)}{   \nu\big(x+\varDelta_\delta\big) } -1 \bigg| 
 +  \bigg| \frac{\mathbf{P}\big( X_s\in  x-y+\varDelta_\delta \big)}{ s\cdot  \nu\big(x-y+\varDelta_\delta\big) } -1 \bigg| \cdot   \frac{\nu\big(x-y+\varDelta_\delta\big)}{   \nu\big(x+\varDelta_\delta\big) } ,
 \eeqnn
 which follows that for large $t$,
 \beqnn
 \lefteqn{ \sup_{0<s\leq t} \sup_{x\geq \theta t} \sup_{|y|\leq Kb(x)}  \sup_{\delta\in[1,2]}\bigg| \frac{\mathbf{P}\big( X_s\in  x-y+\varDelta_\delta \big)}{ s\cdot  \nu\big(x+\varDelta_\delta\big) } -1 \bigg| }\ar\ar\cr
 \ar\leq\ar \sup_{x\geq \theta t}  \sup_{|y|\leq Kb(x)}  \sup_{\delta\in[1,2]} \bigg| \frac{\nu\big(x-y+\varDelta_\delta\big)}{   \nu\big(x+\varDelta_\delta\big) } -1 \bigg|\cr
 \ar\ar + \sup_{0<s\leq t}\sup_{x\geq \theta t}  \sup_{|y|\leq Kb(x)}    \sup_{\delta\in[1,2]}\bigg| \frac{\mathbf{P}\big( X_s\in  x-y+\varDelta_\delta \big)}{ s\cdot  \nu\big(x-y+\varDelta_\delta\big) } -1 \bigg| \cdot \sup_{x\geq \theta t}  \sup_{|y|\leq Kb(x)}  \sup_{\delta\in[1,2]} \frac{\nu\big(x-y+\varDelta_\delta\big)}{   \nu\big(x+\varDelta_\delta\big) } \cr
 \ar\leq\ar \sup_{x\geq \theta t}  \sup_{|y|\leq Kb(x)}  \sup_{\delta\in[1,2]} \bigg| \frac{\nu\big(x-y+\varDelta_\delta\big)}{   \nu\big(x+\varDelta_\delta\big) } -1 \bigg|\cr
 \ar\ar + \sup_{0<s\leq t}\sup_{x\geq \theta t/2}    \sup_{\delta\in[1,2]}\bigg| \frac{\mathbf{P}\big( X_s\in  x+\varDelta_\delta \big)}{ s\cdot  \nu\big(x+\varDelta_\delta\big) } -1 \bigg| \cdot \sup_{x\geq \theta t}  \sup_{|y|\leq Kb(x)}  \sup_{\delta\in[1,2]} \frac{\nu\big(x-y+\varDelta_\delta\big)}{   \nu\big(x+\varDelta_\delta\big) } . 
 \eeqnn
 By Proposition~\ref{Prop.205}, we see that the limit \eqref{main.eqn01} holds for any $K\geq 0$ if and only if it holds with $K=0$, which follows immediately if 
 \beqlb\label{main.eqn02}
 \lim_{t\to\infty} \sup_{0<s\leq T}\sup_{x\geq \theta t}   \sup_{\delta\in[1,2]} \bigg| \frac{\mathbf{P}\big( X_s\in  x+\varDelta_\delta \big)}{ s\cdot  \nu\big(x+\varDelta_\delta\big) } -1 \bigg|   = 0, \quad \forall\, T>0,
 \eeqlb
 and  
 \beqlb\label{main.eqn03}
 \lim_{T\to\infty}  \lim_{t\to\infty}
 \sup_{T<s\leq t}\sup_{x\geq \theta t}   \sup_{\delta\in[1,2]} \bigg| \frac{\mathbf{P}\big( X_s\in  x+\varDelta_\delta \big)}{ s\cdot  \nu\big(x+\varDelta_\delta\big) } -1 \bigg|  =0.
 \eeqlb
 We now prove these two limits. Firstly, the independence between $M$ and $Z$ yields that 
 \beqnn
 \mathbf{P} \big( X_s\in x+\varDelta_\delta \big) 
 \ar=\ar
 \int_{\mathbb{R}} \mathbf{P}\big(  Z_s\in x-y+\varDelta_\delta \big)\, \mathbf{P}(M_s\in dy), 
 \eeqnn
 which, for any $R\in \mathbb{Z}_+$, can be decomposed into the following three parts:
 \beqlb
 I_1(R,s,x,\delta)\ar:=\ar \int_{-R b(x)}^{R b(x)}\mathbf{P}\big(  Z_s\in x-y+\varDelta_\delta \big)\, \mathbf{P}(M_s\in dy) , \label{eqn.305}\\
 I_2(R,s,x,\delta)\ar:=\ar \int_{Rb(x)}^\infty\mathbf{P}\big(  Z_s\in x-y+\varDelta_\delta \big)\, \mathbf{P}(M_s\in dy), \label{eqn.306}\\
 I_3(R,s,x,\delta)\ar:=\ar \int_{-\infty}^{- Rb(x)}\mathbf{P}\big(  Z_s\in x-y+\varDelta_\delta \big)\, \mathbf{P}(M_s\in dy) . \label{eqn.307}
 \eeqlb
 
 We first prove the limit \eqref{main.eqn02} for any $T>0$ fixed. By the triangle inequality, we have 
 \beqlb\label{eqn.308}
 \bigg| \frac{\mathbf{P}\big( X_s\in  x+\varDelta_\delta \big)}{ s\cdot  \nu\big(x+\varDelta_\delta\big) } -1 \bigg| 
 \leq \bigg| \frac{I_1(1,s,x,\delta)}{ s\cdot  \nu\big(x+\varDelta_\delta\big) } -1 \bigg|   
 + 
 \frac{I_2(1,s,x,\delta)}{ s\cdot  \nu\big(x+\varDelta_\delta\big) }  
 +   \frac{I_3(1,s,x,\delta)}{ s\cdot  \nu\big(x+\varDelta_\delta\big) } . 
 \eeqlb
 \begin{enumerate}
 	\item[$\bullet$] From \eqref{eqn.305} with $R=1$, the first term on the right-hand side can be bounded by
 	\beqnn
 	\int_{- b(x)}^{ b(x)} \bigg|\frac{\mathbf{P}\big(  Z_s\in x-y+\varDelta_\delta \big)}{s\cdot  \nu\big(x+\varDelta_\delta\big)}-1\bigg|\, \mathbf{P}(M_s\in dy)  
 	+ \mathbf{P}\big(|M_s|\geq b(x)\big) . 
 	\eeqnn 
 	By Lemma~\ref{Lemma.301} with $K=1$, we have 
 	\beqnn
 	\lefteqn{ \lim_{t\to\infty} \sup_{0<s\leq T}\sup_{x\geq \theta t}  \sup_{\delta\in[1,2]}   \int_{- b(x)}^{ b(x)} \bigg|\frac{\mathbf{P}\big(  Z_s\in x-y+\varDelta_\delta \big)}{s\cdot  \nu\big(x+\varDelta_\delta\big)}-1\bigg|\, \mathbf{P}(M_s\in dy)  }\ar\ar\cr
 	\ar\leq\ar   \lim_{t\to\infty} \sup_{0<s\leq T}\sup_{x\geq \theta t}  \sup_{|y|\leq b(x)}  \sup_{\delta\in[1,2]}  \bigg|\frac{\mathbf{P}\big(  Z_s\in x-y+\varDelta_\delta \big)}{s\cdot  \nu\big(x+\varDelta_\delta\big)}-1\bigg| =0. 
 	\qquad \qquad  
 	\eeqnn
 	Moreover, by the increase of $b$ to infinity, we also have as $t\to \infty$,
 	\beqnn
 	\sup_{s\in(0,T]} \sup_{x\geq \theta t}\mathbf{P}\big(|M_s|\geq b(x)\big) 
 	\ar\leq \ar  \mathbf{P}\Big(  \sup_{s\in(0,T]} |M_s|\geq b(\theta t)\Big) 
 	\to 0. 
 	\eeqnn
 	Putting these two limits together, we get 
 	\beqlb\label{eqn.309}
 	\lim_{t\to\infty} \sup_{0<s\leq T}\sup_{x\geq \theta t}   \sup_{\delta\in[1,2]} \bigg| \frac{I_1(1,s,x,\delta)}{ s\cdot  \nu\big(x+\varDelta_\delta\big) } -1 \bigg|    =0. 
 	\eeqlb
 	
 	\item[$\bullet$]
 	By \eqref{eqn.306}, we have $I_2(1,s,x,\delta) \leq \mathbf{P}\big( M_s>b(x) \big)$  and then 
 	\beqnn
 	\sup_{0<s\leq T}\sup_{x\geq \theta t}   \sup_{\delta\in[1,2]} \frac{I_2(1,s,x,\delta)}{ s\cdot  \nu\big(x+\varDelta_\delta\big) } 
 	\leq \sup_{0<s\leq T}\sup_{x\geq \theta t}   \sup_{\delta\in[1,2]} \frac{\mathbf{P}\big( M_s>b(x) \big)}{s\cdot  \nu\big(x+\varDelta_\delta\big)} .
 	\eeqnn
 	By Proposition~\ref{Prop.202} and Lemma~\ref{Lemma.302}(4), there exist two constants $C,x_0>0$  such that for all $x\geq x_0$,
 	\beqnn
 	\sup_{\delta\in[1,2]} \frac{1}{\nu\big(x+\varDelta_\delta\big)} \leq  \exp\{ C\cdot x^{1-1/\alpha} \}
 	\quad \mbox{and}\quad 
 	\sup_{0<s\leq T} \frac{\mathbf{P}\big( M_s>b(x) \big)}{s } \leq  \exp \Big\{- \frac{1}{2}\cdot b(x)  \log b(x) \Big\},
 	\eeqnn
 	which follows that 
 	\beqnn
 	\sup_{0<s\leq T}\sup_{x\geq \theta t}   \sup_{\delta\in[1,2]} \frac{I_2(1,s,x,\delta)}{ s\cdot  \nu\big(x+\varDelta_\delta\big) }  
 	\leq  \sup_{x\geq \theta t} \exp\big\{ -\frac{1}{2}\cdot b(x)  \log b(x)+ C\cdot x^{1-1/\alpha} \big\} .
 	\eeqnn 
 	Recall the fact that $\alpha \in(1,2]$ and $b(x)\geq C \cdot x^{1/\alpha}$ for large $t$; see \eqref{eqn.Con01}, we have $b(x)  \log b(x) \geq C \cdot x^{1/\alpha}\log x$ and then $b(x)  \log b(x)- x^{1-1/\alpha} \to \infty$ as $x\to\infty$.  
 	This induces that 
 	\beqlb\label{eqn.310}
 	\lim_{t\to\infty}   \sup_{0<s\leq T}\sup_{x\geq \theta t}   \sup_{\delta\in[1,2]} \frac{I_2(1,s,x,\delta)}{ s\cdot  \nu\big(x+\varDelta_\delta\big) }   =0. 
 	\eeqlb
 	
 	\item[$\bullet$] 
 	We now consider the last term in \eqref{eqn.308}. By  \eqref{eqn.307}, we have 
 	\beqlb \label{eqn.311}
 	\frac{I_3(1,s,x,\delta)}{ s\cdot  \nu\big(x+\varDelta_\delta\big) }  
 	\ar=\ar   \int_{-\infty}^{- b(x)}\frac{\mathbf{P}\big(  Z_s\in x-y+\varDelta_\delta \big)}{s\cdot  \nu\big(x-y+\varDelta_\delta\big) }\cdot \frac{  \nu\big(x-y+\varDelta_\delta\big)}{  \nu\big(x+\varDelta_\delta\big)}\, \mathbf{P}(M_s\in dy),
 	\eeqlb
 	which along with the uniform upper bound \eqref{eqn.3011} induces that  as $t \to\infty$,
 	\beqlb\label{eqn.313}
 	\sup_{0<s\leq T} \sup_{x\geq \theta t}   \sup_{\delta\in[1,2]}\frac{I_3(1,s,x,\delta)}{ s\cdot  \nu\big(x+\varDelta_\delta\big) }  \leq C\cdot   \sup_{x\geq \theta t} \sup_{0<s\leq T} \mathbf{P}\big( M_s \leq -b(x) \big) \to 0 ,
 	\eeqlb
 	since $\inf_{s\in[0,T]} M_s>-\infty$ a.s. and $\inf_{x\geq\theta t }b(x)\to\infty$. 
 \end{enumerate}
 The  limit \eqref{main.eqn02} follows immediately by taking \eqref{eqn.309}, \eqref{eqn.310} and \eqref{eqn.313} back into  \eqref{eqn.308}.

 We now turn to prove the limit \eqref{main.eqn03}. 
 For any large $R\in \mathbb{Z}_+$, since $b(z)=o(z)$ as $z\to\infty$, we see that $Rb(t)\leq \theta t /2$ for all $t>0$ large enough. 
 Without loss of generality, we always assume that  
 \beqnn
 t>T>0
 \quad \mbox{and}\quad
 Rb(t)\leq \theta t /2.
 \eeqnn
 We first decompose $I_2(R,s,x,\delta)$ into the following two parts:
 \beqlb
 I_{21}(R,s,x,\delta)\ar:=\ar \int_{R b(x)}^{x/2}\mathbf{P}\big(  Z_s\in x-y+\varDelta_\delta \big)\, \mathbf{P}(M_s\in dy), \label{eqn.3061}\\
 I_{22}(R,s,x,\delta)\ar:=\ar \int_{x/2}^\infty \mathbf{P}\big(  Z_s\in x-y+\varDelta_\delta \big)\, \mathbf{P}(M_s\in dy). \label{eqn.3062} 
 \eeqlb
 Consequently, in view of \eqref{eqn.308} we have for all $R\in \mathbb{Z}_+$ and large $t>T$, 
 \beqlb\label{eqn.316}
 \bigg| \frac{\mathbf{P}\big( X_s\in  x+\varDelta_\delta \big)}{ s\cdot  \nu\big(x+\varDelta_\delta\big) } -1 \bigg| 
 \ar\leq\ar \bigg| \frac{I_1(R,s,x,\delta)}{ s\cdot  \nu\big(x+\varDelta_\delta\big) } -1 \bigg|   
 + 
 \frac{I_{21}(R,s,x,\delta)}{ s\cdot  \nu\big(x+\varDelta_\delta\big) } \cr
 \ar\ar + \frac{I_{22}(R,s,x,\delta)}{ s\cdot  \nu\big(x+\varDelta_\delta\big) }  + \frac{I_3(R,s,x,\delta)}{ s\cdot  \nu\big(x+\varDelta_\delta\big) } .   
 \eeqlb
 \begin{enumerate}
 	\item[$\bullet$] For the first term on the right-hand side, by \eqref{eqn.305} we have 
 	\beqnn
 	\bigg| \frac{I_1(R,s,x,\delta)}{ s\cdot  \nu\big(x+\varDelta_\delta\big) } -1 \bigg|
 	\ar\leq\ar \int_{-R b(x)}^{R b(x)}\bigg| \frac{\mathbf{P}\big(  Z_s\in x-y+\varDelta_\delta \big)}{s\cdot  \nu\big(x+\varDelta_\delta\big)}-1 \bigg|\, \mathbf{P}(M_s\in dy) + \mathbf{P}\big(|M_s|\geq Rb(x)\big),
 	\eeqnn
 	which follows that 
 	\beqnn
 	\lefteqn{\sup_{T<s\leq t}\sup_{x\geq \theta t}   \sup_{\delta\in[1,2]}   \bigg| \frac{I_1(R,s,x,\delta)}{ s\cdot  \nu\big(x+\varDelta_\delta\big) } -1 \bigg| }\ar\ar\cr
 	\ar\leq\ar \sup_{T<s\leq t}\sup_{x\geq \theta t} \sup_{|y|\leq Rb(x)}  \sup_{\delta\in[1,2]} \bigg| \frac{\mathbf{P}\big(  Z_s\in x-y+\varDelta_\delta \big)}{s\cdot  \nu\big(x+\varDelta_\delta\big)}-1 \bigg| + \sup_{T<s\leq t}\sup_{x\geq \theta t}  \mathbf{P}\big(|M_s|\geq Rb(x)\big). 
 	\eeqnn
 	The first term on the right-hand side vanishes as $t\to\infty$; see Lemma~\ref{Lemma.301}. 
 	For the second term, by using Proposition~\ref{Prop.204} and then Lemma~\ref{Lemma.302}(1), we have  as  $T\to\infty$ and then $R\to\infty$,
 	\beqlb\label{eqn.317}
 	\sup_{T<s\leq t}\sup_{x\geq \theta t}  \mathbf{P}\big(|M_s|\geq Rb(x)\big) 
 	\leq \sup_{s>T}  \mathbf{P}\big(|M_s|\geq CRb(s)\big) \to 0.
 	\eeqlb
  In conclusion, we have  
 	\beqlb\label{eqn.318}
 	\lim_{R\to\infty}\lim_{T\to\infty}\limsup_{t\to\infty} \sup_{T<s\leq t}\sup_{x\geq \theta t}   \sup_{\delta\in[1,2]}   \bigg| \frac{I_1(R,s,x,\delta)}{ s\cdot  \nu\big(x+\varDelta_\delta\big) } -1 \bigg| =0.
 	\eeqlb
 	
 	\item[$\bullet$]
 	For the second term on the right-hand side of \eqref{eqn.316}, by \eqref{eqn.3061} we have 
 	\beqnn
 	\frac{I_{21}(R,s,x,\delta)}{ s\cdot  \nu\big(x+\varDelta_\delta\big) } 
 	\ar=\ar   \int_{R b(x)}^{x/2} \frac{\mathbf{P}\big(  Z_s\in x-y+\varDelta_\delta \big)}{s\cdot  \nu\big(x-y+\varDelta_\delta\big) }\cdot \frac{\nu\big(x-y+\varDelta_\delta\big) }{  \nu\big(x+\varDelta_\delta\big) }\, \mathbf{P}(M_s\in dy). \qquad
 	\eeqnn
 	By \eqref{eqn.301} with $K=0$, there exists a constant $t_0>0$ such that for all $t\geq t_0$,
 	\beqnn
 	\sup_{0<s\leq t}\sup_{x\geq \theta t}   \sup_{\delta\in[1,2]} \sup_{R b(x)<y\leq x/2}\frac{\mathbf{P}\big(  Z_s\in x-y+\varDelta_\delta \big)}{s\cdot  \nu\big(x-y+\varDelta_\delta\big) }
 	\leq \sup_{0<s\leq t}\sup_{x\geq \theta t/2}   \sup_{\delta\in[1,2]}  \frac{\mathbf{P}\big(  Z_s\in x+\varDelta_\delta \big)}{s\cdot  \nu\big(x+\varDelta_\delta\big) } \leq 2, 
 	\eeqnn
 	which follows that 
 	\beqlb\label{eqn.319}
 	\sup_{T<s\leq t}\sup_{x\geq \theta t}   \sup_{\delta\in[1,2]}   \frac{I_{21}(R,s,x,\delta)}{ s\cdot  \nu\big(x+\varDelta_\delta\big) } 
 	\ar\leq \ar
 	2\cdot   \sup_{T<s\leq t}\sup_{x\geq \theta t}   \sup_{\delta\in[1,2]}     \int_{R b(x)}^{x/2} \frac{\nu\big(x-y+\varDelta_\delta\big) }{  \nu\big(x+\varDelta_\delta\big) }\, \mathbf{P}(M_s\in dy).\quad 
 	\eeqlb
 	For any constant $\eta>0$ to be specified later, by Proposition~\ref{Prop.203}  there exist two constants $C, t_0>0$ such that for all $t\geq t_0$, $x\geq \theta t$ and $Rb(x)<y\leq x/2$, 
 	\beqnn
 	\sup_{\delta\in[1,2]} \frac{  \nu\big(x-y+\varDelta_\delta\big)}{ \nu\big(x+\varDelta_\delta\big)} \leq C\cdot \exp \Big\{ \eta \cdot \frac{y}{b(x)} \Big\} . 
 	\eeqnn
 	Taking these two estimates back into \eqref{eqn.319}, we have that for all $t\geq t_0$, 
 	\beqnn
 	\sup_{T<s\leq t}\sup_{x\geq \theta t}   \sup_{\delta\in[1,2]}   \frac{I_{21}(R,s,x,\delta)}{ s\cdot  \nu\big(x+\varDelta_\delta\big) } 
 	\ar\leq \ar C\cdot \sup_{T<s\leq t}\sup_{x\geq \theta t}  \int_{Rb(x)}^{x/2}\exp \Big\{ \eta \cdot \frac{y}{b(x)} \Big\} \, \mathbf{P}(M_s\in dy)\cr
 	\ar \leq\ar  C\cdot \sup_{T<s\leq t}\sup_{x\geq \theta t} \sum_{k=R}^\infty \int_{k b(x)}^{(k+1)b(x)}\exp \Big\{ \eta \cdot \frac{y}{b(x)} \Big\} \, \mathbf{P}(M_s\in dy) \cr
 	\ar\leq\ar C  \sum_{k=R}^\infty e^{ \eta \cdot (k+1)} \cdot   \sup_{T<s\leq t} \sup_{x\geq \theta t} \mathbf{P}\big(M_s\geq kb(x)\big) .
 	\eeqnn
 	By Chebyshev's inequality, we have for all $k\geq R$,
 	\beqnn
 	\mathbf{P}\big(M_s\geq kb(x)\big) 
 	\leq     \mathbf{E}\Big[ \exp\Big\{\frac{M_s}{b(s)} \Big\} \Big] \cdot \exp\Big\{-k\cdot     \frac{b(x)}{b(s)}  \Big\} .
 	\eeqnn
 	By Proposition~\ref{Prop.204} and Lemma~\ref{Lemma.302}(2),  there exist  constants $C,c_0>0$ such that for large $T\geq t_0$, 
 	\beqnn
 	\sup_{s\geq T} \mathbf{E}\Big[ \exp\Big\{\frac{M_s}{b(s)} \Big\} \Big] \leq C
 	\quad \mbox{and}\quad \inf_{t\geq T} \inf_{0<s\leq t} \inf_{x\geq \theta t}  \frac{b(x)}{b(s)} \geq c_0,
 	\eeqnn
 	which follows that  $
 	\sup_{T<s\leq t} \sup_{x\geq \theta t} \mathbf{P}\big(M_s\geq kb(x)\big)
 	\leq C\cdot e^{-c_0\cdot k}$
 	and 
 	\beqnn
 	\sup_{T<s\leq t}\sup_{x\geq \theta t}   \sup_{\delta\in[1,2]}   \frac{I_{21}(R,s,x,\delta)}{ s\cdot  \nu\big(x+\varDelta_\delta\big) }   \leq C \sum_{k=R}^\infty e^{ \eta \cdot (k+1)-c_0\cdot k} <\infty,
 	\eeqnn
 	which  vanishes as $R\to\infty$ if we choose $0<\eta <c_0$. Consequently,  
 	\beqlb\label{eqn.3181}
 	\lim_{R\to\infty}\lim_{T\to\infty} \limsup_{t\to\infty}    \sup_{T<s\leq t}\sup_{x\geq \theta t}   \sup_{\delta\in[1,2]}   \frac{I_{21}(R,s,x,\delta)}{ s\cdot  \nu\big(x+\varDelta_\delta\big) }  =0. 
 	\eeqlb
 	
 	\item[$\bullet$] For the third term on the right-hand side of \eqref{eqn.316}, by \eqref{eqn.3062} we always have 
 	\beqlb\label{eqn.220}
 	\frac{I_{22}(R,s,x,\delta)}{ s\cdot  \nu\big(x+\varDelta_\delta\big) } \leq  \frac{\mathbf{P}(M_s\geq x/2)}{s\cdot  \nu\big(x+\varDelta_\delta\big)}  . 
 	\eeqlb
 	Firstly, by Proposition~\ref{Prop.202} there exist two constants $C_0,x_0>0$ such that for all $x\geq x_0$,
 	\beqnn
 	\inf_{\delta\in[1,2]}   \nu\big(x+\varDelta_\delta\big)\geq \exp\{-C_0\cdot x^{1-1/\alpha}\}.
 	\eeqnn
 	Moreover, by Lemma~\ref{Lemma.302}(3)  with $\eta =\theta/2$,  there exist some constants $C_\theta,c_\theta>0$ such that $\mathbf{P}(M_s\geq x/2) \leq C_\theta\cdot \exp\{-c_\theta x\} $ for all $x\geq \theta t\geq \theta s$. 
 	Taking these two estimates back into \eqref{eqn.220}, 
 	\beqlb\label{eqn.321}
 	\lim_{t\to\infty}\sup_{T<s\leq t}\sup_{x\geq \theta t}   \sup_{\delta\in[1,2]}  \frac{I_{22}(R,s,x,\delta)}{ s\cdot  \nu\big(x+\varDelta_\delta\big) } 
 	\leq \frac{C}{T} \lim_{t\to\infty} \sup_{x\geq \theta t}  \exp\{-c_\theta x+C_0\cdot x^{1-1/\alpha}\}  =0.
 	\eeqlb

 	\item[$\bullet$]
 	For the last term on the right-hand side of \eqref{eqn.316}, similarly as in \eqref{eqn.311}-\eqref{eqn.313} we also have  
 	\beqlb \label{eqn.322}
 	\sup_{T<s\leq t} \sup_{x\geq \theta t}   \sup_{\delta\in[1,2]}\frac{I_3(R,s,x,\delta)}{ s\cdot  \nu\big(x+\varDelta_\delta\big) }  
 	\ar\leq\ar C\cdot   \sup_{x\geq \theta t} \sup_{T<s\leq t} \mathbf{P}\big( M_s \leq -R b(x) \big) \cr
 	\ar\leq\ar C\cdot   \sup_{x\geq \theta t} \sup_{T<s\leq t} \mathbf{P}\big( |M_s| \geq R b(x) \big) \to 0,
 	\eeqlb 
 	as $t\to\infty$, then $T\to\infty$ and finally $R\to\infty$; see \eqref{eqn.317}. 
 \end{enumerate}
 The limit \eqref{main.eqn03} follows by taking \eqref{eqn.318}, \eqref{eqn.3181}  \eqref{eqn.321} and \eqref{eqn.322} back into \eqref{eqn.316}. 
 In conclusion, the limit \eqref{main.eqn01} holds. 
 
 \textbf{Step 2.} We now prove the limit \eqref{main.eqn} with the help of \eqref{main.eqn01}.   
 For any $s,x\geq 0$ and $\delta\geq 1$, by the triangle inequality,
 \beqnn
 \lefteqn{\big|\mathbf{P}\big( X_s\in x-y+\varDelta_\delta \big) -s\cdot \nu(x+\varDelta_\delta)\big| }\ar\ar\cr
 \ar\leq\ar  \sum_{k=0}^{[\delta]-2} \big| \mathbf{P}\big( X_s\in x-y+k+\varDelta_1 \big) -s\cdot \nu(x+k+\varDelta_1) \big|\cr\cr
 \ar\ar  + \Big|\mathbf{P}\big( X_s\in x-y+[\delta]-1+\varDelta_{\delta-[\delta]+1} \big) -s\cdot \nu\big( x+[\delta]-1+\varDelta_{\delta-[\delta]+1} \big) \Big|
 \eeqnn
 with the convention that the sum equals  zero when $[\delta]=1$.
 The limit \eqref{main.eqn01} tells us that there exists a positive function $\varepsilon_t$ decreasing to $0$ at infinity such that for all $t>0$, 
 \beqnn
 \sup_{0<s\leq t}\sup_{x\geq \theta t} \sup_{|y|\leq Kb(x)}  \sup_{\delta'\in[1,2]} \bigg| \frac{\mathbf{P}\big( X_s\in  x-y+\varDelta_{\delta'} \big)}{ s\cdot  \nu\big(x+\varDelta_{\delta'}\big) } -1 \bigg| \leq \varepsilon_t,
 \eeqnn
 which follows that for all $k\geq 0$,
 \beqnn
 \sup_{0<s\leq t}\sup_{x\geq \theta t} \sup_{|y|\leq Kb(x)}  \sup_{\delta'\in[1,2]} \bigg| \frac{\mathbf{P}\big( X_s\in  x-y+k+\varDelta_{\delta'} \big)}{ s\cdot  \nu\big(x+k+\varDelta_{\delta'}\big) } -1 \bigg| \leq \varepsilon_t.
 \eeqnn
 Consequently, we have for all $0< s\leq t$, $x\geq \theta t$ and $\delta \geq 1$, 
 \beqnn
 \Big|  \mathbf{P}\big( X_s\in x-y+\varDelta_\delta \big) -s\cdot \nu(x+\varDelta_\delta)  \Big| 
 \leq  \varepsilon_t \cdot\bigg(\sum_{k=0}^{[\delta]-2} s\cdot  \nu\big(x+k+\varDelta_1\big) +s\cdot \nu\big(x+[\delta]-1+\varDelta_{\delta-[\delta]+1}\big)  \bigg),
 \eeqnn
 which equals 	$\varepsilon_t\cdot s\cdot \nu(x+\varDelta_\delta)$. This yields that 
 \beqnn
 \sup_{0<s\leq t}\sup_{x\geq \theta t}  \sup_{|y|\leq Kb(x)}  \sup_{\delta\in[1,\infty]}  \bigg| \frac{\mathbf{P}\big( X_s\in  x-y+\varDelta_\delta \big)}{ s\cdot  \nu\big(x+\varDelta_\delta\big) } -1 \bigg| \leq \varepsilon_t,
 \eeqnn
 which goes to $0$ as $t\to\infty$ and then the limit \eqref{main.eqn} holds.
 The proof ends. 
 \qed 
 
 \begin{proposition}\label{lem.ConditionEquivalence}
 	For $\alpha>1$, we have $x\mapsto\nu\big(x+\varDelta_\delta\big)
 	\in \mathrm{RV}^\infty_{-\alpha-1}$ for any 
 	$\delta>0$ if and only if $\overline{\nu}\in\mathrm{RV}_{-\alpha}^{\infty}$
 	and,  as $x\to\infty$,
 	\beqlb\label{eqn.30111}
 	\nu\big(x+\varDelta_\delta\big)
 	\sim
 	\alpha\cdot\frac{\overline{\nu}(x)}{x}\cdot\delta . 
 	\eeqlb
 \end{proposition}
 \proof
 	Fix $\delta>0$. Since
 	$
 	(x,\infty)
 	=
 	\bigcup_{k=0}^{\infty}
 	\big(x+k\delta,x+(k+1)\delta\big],
 	$
 	we have
 	\beqlb\label{eqn.ConditionEquivalence01}
 	\frac{\delta\overline{\nu}(x)}
 	{x\nu\big(x+\varDelta_\delta\big)}
 	=
 	\frac{\delta}{x}
 	\sum_{k=0}^{\infty}
 	\frac{\nu\big(x+k\delta+\varDelta_\delta\big)}
 	{\nu\big(x+\varDelta_\delta\big)}.
 	\eeqlb
 	For fixed $M>0$, an application of the uniform
 	convergence theorem for regularly varying functions (see e.g. Proposition~0.5 in \cite[p.17]{Resnick1987}), together with the assumption $x\mapsto\nu(x+\varDelta_\delta)\in\mathrm{RV}_{-\alpha-1}^\infty$, gives, as $x\to\infty$, 
 	\beqnn
 	\sup_{0\leq k\leq[Mx/\delta]}
 	\bigg|
 	\frac{\nu\big(x+k\delta+\varDelta_\delta\big)}
 	{\nu\big(x+\varDelta_\delta\big)}
 	-
 	\Big(1+\frac{k\delta}{x}\Big)^{-\alpha-1}
 	\bigg|
 	\to 0.
 	\eeqnn
 	Therefore, the Riemann-sum approximation yields as $x\to\infty$,
 	\beqlb\label{eqn.ConditionEquivalence02}
 	\frac{\delta}{x}
 	\sum_{k=0}^{[Mx/\delta]}
 	\frac{\nu\big(x+k\delta+\varDelta_\delta\big)}
 	{\nu\big(x+\varDelta_\delta\big)}
 	\to 
 	\int_0^M(1+s)^{-\alpha-1}\,ds.
 	\eeqlb
 	For the remainder, fix $\varepsilon\in(0,\alpha)$. By Potter's bound; see Theorem~1.5.6 in \cite[p.25]{BinghamGoldieTeugels1987}, there
 	exist constants $C,x_0>0$ such that, for all $x\geq x_0$ and $k\geq0$,
 	\beqnn
 	\frac{\nu\big(x+k\delta+\varDelta_\delta\big)}
 	{\nu\big(x+\varDelta_\delta\big)}
 	\leq
 	C\Big(1+\frac{k\delta}{x}\Big)^{-\alpha-1+\varepsilon}.
 	\eeqnn
 	Since $\alpha+1-\varepsilon>1$, an integral comparison gives
 	\beqnn
 	\limsup_{x\to\infty}
 	\frac{\delta}{x}
 	\sum_{k>[Mx/\delta]}
 	\frac{\nu\big(x+k\delta+\varDelta_\delta\big)}
 	{\nu\big(x+\varDelta_\delta\big)}
 	\leq 
 	C\int_M^\infty(1+s)^{-\alpha-1+\varepsilon}\,ds,
 	\eeqnn
 	which goes to $0$ as $M\to\infty$. 
 	Combining this estimate with \eqref{eqn.ConditionEquivalence01} and
 	\eqref{eqn.ConditionEquivalence02}, and  letting first $x\to\infty$
 	and then $M\to\infty$, we obtain
 	\beqnn
 	\frac{\delta\overline{\nu}(x)}
 	{x\cdot \nu\big(x+\varDelta_\delta\big)}
 	\to 
 	\int_0^\infty(1+s)^{-\alpha-1}\,ds
 	=
 	\frac1\alpha 
 	\eeqnn
  which yields the asymptotic equivalence  \eqref{eqn.30111} and  $\overline{\nu}\in\mathrm{RV}_{-\alpha}^{\infty}$.   
  Conversely, suppose that $\overline{\nu}\in\mathrm{RV}_{-\alpha}^\infty$ and that \eqref{eqn.30111} holds. Then, for any $\lambda>0$ and sufficiently large $x$,
  \begin{equation}
  	\frac{\nu(\lambda x+\varDelta_\delta)}{\nu(x+\varDelta_\delta)}
  	\sim
  	\frac{\alpha\,\overline{\nu}(\lambda x)/(\lambda x)\cdot\delta}
  	{\alpha\,\overline{\nu}(x)/x\cdot\delta}
  	=
  	\frac{\overline{\nu}(\lambda x)}{\lambda\,\overline{\nu}(x)}
  	\to
  	\lambda^{-\alpha-1},
  \end{equation}
  since $\overline{\nu}\in\mathrm{RV}_{-\alpha}^\infty$. Hence $x\mapsto\nu(x+\varDelta_\delta)\in\mathrm{RV}_{-\alpha-1}^\infty$. This completes the proof.
 \qed

 \medskip
 
 \textit{\textbf{Proof of Corollary~\ref{MainCorollary}.}} By Proposition~\ref{lem.ConditionEquivalence}, one easily verifies that Condition~\ref{MainCondition01} is satisfied by $X$ with $b$ given in \eqref{eqn.002}. Hence the limit \eqref{main.eqn} holds. 
 By the triangle inequality, 
  \beqnn 
  \lefteqn{\bigg| \frac{\mathbf{P}\big(X_s  \in x -y+ \varDelta_\delta \big)}{s\cdot  \nu\big(x + \varDelta_\delta \big)} -1\bigg| }\ar\ar\cr
  \ar\leq\ar \frac{  \nu\big(x -y+ \varDelta_\delta \big)   }{   \nu\big(x + \varDelta_\delta \big)}\cdot \bigg| \frac{\mathbf{P}\big(X_s   \in x -y+ \varDelta_\delta \big)}{s\cdot  \nu\big(x-y + \varDelta_\delta \big) } -1\bigg|   + \bigg|  \frac{  \nu\big(x-y + \varDelta_\delta \big)  }{   \nu\big(x + \varDelta_\delta \big)}-1 \bigg|.
  \eeqnn
 Using the uniform convergence theorem for regularly varying functions again,  we obtain, as $t\to\infty$, 
  \beqnn
  \sup_{x\geq \lambda t} \sup_{|y|\leq \vartheta_x} \bigg|  \frac{  \nu\big(x-y + \varDelta_\delta \big)  }{ \nu\big(x + \varDelta_\delta \big) }-1 \bigg| \to 0,
  \eeqnn
  This, along with \eqref{main.eqn} with $K=0$, induces that  
  \beqnn
   \lefteqn{ \limsup_{t\to\infty} \sup_{0<s\leq  t} \sup_{x\geq \lambda t} \sup_{|y|\leq \vartheta_x} \sup_{\delta\in[\delta_0,\infty]} 
 	 \bigg| \frac{\mathbf{P}\big(X_s   \in x -y+ \varDelta_\delta \big)}{s\cdot  \nu\big(x + \varDelta_\delta \big) } -1\bigg|  }\ar\ar\cr
  \ar\leq \ar  \limsup_{t\to\infty} \sup_{0<s\leq  t} \sup_{x\geq \lambda t} \sup_{|y|\leq \vartheta_x} \sup_{\delta\in[\delta_0,\infty]}  \bigg| \frac{\mathbf{P}\big(X_s   \in x -y+ \varDelta_\delta \big)}{s\cdot  \nu\big(x-y + \varDelta_\delta \big) } -1\bigg|   =0.
  \eeqnn  
 This completes the proof. 
  \qed 
 
 \subsection{Proof of Lemma~\ref{Lemma.302}}

 \textbf{\textit{Proof of Lemma~\ref{Lemma.302}(1).}}  
 Recall that $X_t=M_t+Z_t$, we have  
 \beqnn 
 \mathbf P\big( |M_t|\geq Kb(t) \big) \leq  
 \mathbf P\big(|X_t|\geq Kb(t)/2\big)+ \mathbf P\big(|Z_t|\geq Kb(t)/2\big) . 
 \eeqnn
 The first probability on the right-hand side vanishes as $t\to\infty $ and then $K\to\infty$; see \eqref{eqn.Con01}. For the second one, by using Chebyshev's inequality we have 
 \beqlb\label{eqn.339}
 \mathbf{P}\big( |Z_t|\geq Kb(t) /2\big) 
 \leq \frac{2^\alpha\cdot \mathbf{E}\big[ |Z_t|^\alpha \big]}{|Kb(t)|^\alpha}.
 \eeqlb
 By the Burkholder-Davis-Gundy inequality, there exists a constant $C>0$ such that  for all $t\geq 0$,
 \beqnn
 \mathbf{E}\big[ |Z_t|^\alpha \big] 
 \ar \leq\ar C \cdot  \mathbf E\Big[\Big( \int_0^t\int_1^\infty  y^2\,N(ds,dy)\Big)^{\alpha/2} \Big].
 \eeqnn 
 The fact that $N\big((0,t],(1,\infty)\big) <\infty$ a.s. and the inequality $(x+y)^{\alpha/2}\leq x^{\alpha/2}+ y^{\alpha/2}$ for all $x,y\geq 0$ induce that uniformly in $t\geq 0$, 
 \beqnn
 \mathbf{E}\big[ |Z_t|^\alpha \big]  
 \ar\leq\ar C \cdot  \mathbf E\Big[ \int_0^t\int_1^\infty  y^\alpha\,N(ds,dy) \Big]
 = C\cdot t\cdot \int_1^\infty  y^\alpha\,\nu(dy) 
 \leq C\cdot t.
 \eeqnn 
 Taking this back into \eqref{eqn.339} and then using  the inequality $b(t)\geq C\cdot t^{1/\alpha}$ for large $t\geq 0$,  
 \beqnn
 \mathbf{P}\big( |Z_t|\geq Kb(t)/2 \big)  \leq  \frac{C\cdot t}{|Kb(t)|^\alpha} 
 \leq \frac{C}{K^\alpha}. 
 \eeqnn 
 which goes to $0$ as $K\to\infty$. Here we have proved Lemma~\ref{Lemma.302}(1).
 \qed

 \begin{proposition} \label{Prop.306}
 	There exists a constant $t_0>0$ such that 
 	\beqnn
 	\sup_{t\geq t_0}\frac{t}{b(t)^2} +	\sup_{t\geq t_0} \, t \cdot \int_{-\infty}^1  \Big|\frac{y}{b(t)}\Big|^2 \wedge \Big|\frac{y}{b(t)} \Big|\, \nu(dy) <\infty. 
 	\eeqnn 
 \end{proposition}
 \proof The assumptions $\alpha \in(1,2]$ and
 $\limsup_{s\to\infty}s/b(s)^\alpha<\infty$ directly yield that for some $t_0>0$,
 \beqlb\label{eqn.336}
 \sup_{t\geq t_0}\frac{t}{b(t)^2} <\infty. 
 \eeqlb
 To bound the second supremum, it suffices to find some constant $t_0>0$ such that  
 \beqlb\label{eqn.340}
 \sup_{t\geq t_0}\,  t \cdot \int_{-\infty}^1 1\wedge  \Big(\frac{y}{b(t)}\Big)^2  \nu(dy) <\infty
 \quad \mbox{and}\quad
 \sup_{t\geq t_0}\, t \cdot \int_{-\infty}^{-b(t)}  \frac{|y|}{b(t)} \nu(dy) <\infty. 
 \eeqlb
 
 By Lemma~\ref{Lemma.302}(1), we first choose $K$ and $t_0$ large enough such that 
 \beqnn
 \sup_{t\geq t_0}  \mathbf P\big( |M_t|\geq Kb(t) \big) \leq \frac{1}{8}.
 \eeqnn 
 By the inequality $|1-e^{\mathtt{i}z}|\leq |z|\wedge 2$ for all $z\in \mathbb{R}$, we have for all $t\geq t_0$ and $|u|\leq \frac{1}{4K}$, 
 \beqnn
 \bigg| 1- \mathbf{E}\Big[ \exp\Big\{ \mathtt{i}u \cdot \frac{M_t}{b(t)} \Big\} \Big] \bigg|
 \ar\leq\ar  \mathbf{E}\bigg[ \Big|1- \exp\Big\{ \mathtt{i}u \cdot \frac{M_t}{b(t)} \Big\} \Big|\cdot \mathbf{1}_{|M_t|\leq  Kb(t)} \bigg] \cr
 \ar\ar+ \mathbf{E}\bigg[ \Big|1- \exp\Big\{ \mathtt{i}u \cdot \frac{M_t}{b(t)} \Big\} \Big|\cdot \mathbf{1}_{|M_t|\geq  Kb(t)} \bigg]\cr
 \ar\leq\ar \mathbf{E}\bigg[ |u| \cdot \frac{|M_t|}{b(t)}  \cdot \mathbf{1}_{|M_t|\leq  Kb(t)} \bigg] 
 + 2\mathbf{P}\big( |M_t|\geq  Kb(t) \big) \leq \frac{1}{2},
 \eeqnn
 which, along with the triangle inequality, induces that 
 \beqnn
 \inf_{t\geq t_0}\inf_{|u|\leq \frac{1}{4K}} \bigg|  \mathbf{E}\Big[ \exp\Big\{ \mathtt{i}u \cdot \frac{M_t}{b(t)} \Big\} \Big] \bigg|\geq \frac{1}{2}. 
 \eeqnn
 Recall the characteristic function of $M$, i.e.,
 \beqnn
 \mathbf{E}\Big[ \exp\Big\{ \mathtt{i}u \cdot \frac{M_t}{b(t)} \Big\} \Big] 
 =  \exp\bigg\{- \frac{\sigma^2}{2} \cdot |u|^2\cdot \frac{ t}{|b(t)|^2} + t\cdot \int_{-\infty}^1 \Big( \exp\Big\{\mathtt{i}\frac{u}{b(t)}y\Big\} -1- \mathtt{i}\frac{u}{b(t)}y  \Big) \nu(dy)\bigg\},
 \eeqnn
 which follows that 
 \beqnn
 -\log \bigg|  \mathbf{E}\Big[ \exp\Big\{ \mathtt{i}u \cdot \frac{M_t}{b(t)} \Big\} \Big] \bigg|
 \ar=\ar  \frac{\sigma^2}{2} \cdot |u|^2\cdot \frac{ t}{|b(t)|^2} +t\cdot \int_{-\infty}^1 \Big(
 1- \cos\Big( \frac{u}{b(t)}\cdot y \Big)\Big)\nu(dy)\leq \ln 2,
 \eeqnn
 for all $t\geq t_0$ and $|u|\leq \frac{1}{4K}$. 
 Hence, for any $0<u_0\leq \frac{1}{4K}$, we have 
 \beqnn
 \frac{t}{u_0} \int_0^{u_0} du\int_{-\infty}^1 \Big(
 1- \cos\Big( \frac{u}{b(t)}\cdot y \Big)\Big)\nu(dy) \leq  \log2 ,
 \eeqnn
 which along with Fubini's theorem yields that 
 \beqnn
 t\cdot \int_{-\infty}^1  \Big(1 - \frac{b(t)}{u_0y}\sin\Big( \frac{u_0 y}{b(t)} \Big)\Big)\nu(dy)\leq  \ln 2 .
 \eeqnn
 Note that $1-\sin(z)/z \geq |z|^2/12$ for all $0<|z|\leq \pi$ and $1-\sin(z)/z\geq 2/3$ for all $|z|\geq \pi$, we have for some constant $C>0$,
 \beqnn
 1 - \frac{b(t)}{u_0y}\sin\Big( \frac{u_0 y}{b(t)}\Big)  \geq C\cdot \frac{\big(u_0y/b(t)\big)^2}{1+ \big(u_0y/b(t)\big)^2} \geq C\cdot \frac{|y/b(t)|^2}{1+|y/b(t)|^2},
 \eeqnn
 uniformly in $y\neq 0$ and $t\geq t_0$, which induces that 
 \beqnn
 \sup_{t\geq t_0}\, t\cdot \int_{-\infty}^1  \frac{|y/b(t)|^2}{1+|y/b(t)|^2}\, \nu(dy) \leq \frac{\ln 2}{C} .
 \eeqnn
 This along with the inequality $2\frac{|z|}{1+|z|} \geq |z|\wedge 1$ for all $z\in \mathbb{R}$ induces that
 \beqlb\label{eqn.342}
 \sup_{t\geq t_0} \, t \cdot \int_{-\infty}^1 1\wedge  \Big(\frac{y}{b(t)}\Big)^2  \nu(dy) \leq \frac{2\ln 2}{C} .
 \eeqlb
 
 We now start to prove the second inequality in \eqref{eqn.340}.  
 By \eqref{eqn.336} and \eqref{eqn.342}, we choose a large constant $C_0>0$ such that for all $t\geq t_0$,
 \beqnn
 \frac{\sigma^2\cdot t}{|b(t)|^2} +  t\cdot\int_{-b(t)}^1   \Big|\frac{y}{b(t)}\Big|^2   \,\nu(dy)+ t\cdot \nu\big((-\infty,-b(t)]\big)=\frac{\sigma^2\cdot t}{|b(t)|^2} +  t\cdot\int_{-\infty}^1 1\wedge \Big|\frac{y}{b(t)}\Big|^2 \,\nu(dy) \leq C_0. 
 \eeqnn 
 Moreover, by Lemma~\ref{Lemma.302}(1), we choose  $K$ and $t_0$ large enough again such that 
 \beqlb\label{eqn.343}
 \sup_{t\geq t_0} \mathbf{P}\Big( \frac{M_t}{b(t)}\geq K \Big) \leq \frac{e^{-C_0}}{4}. 
 \eeqlb
 The process $M$ can be decomposed into the following two parts:
 \beqnn 
 Y_t:= \sigma \cdot B_t + \int_0^t\int_{-b(t)}^1 y\,\widetilde{N}(ds,dy) - t\cdot \int_{-\infty}^{-b(t)} y\,\nu(dy) 
 \quad \mbox{and}\quad 
 J_t:= \int_0^t\int_{-\infty}^{-b(t)} y\,N(ds,dy)  ,
 \eeqnn
 which are independent because of the orthogonality of $N(ds,dy)$. 
 Notice that $J_t \leq 0$ a.s. and for all  $t\geq t_0$, 
 \beqnn
 \mathbf{P}\big( J_t=0 \big) =\mathbf{P}\big(N\big((0,t], (-\infty, -b(t)]\big)=0 \big) =\exp\big\{ - t\cdot \nu\big((-\infty,-b(t)]\big)\big\} \geq e^{-C_0},
 \eeqnn
 we have 
 \beqnn
 \mathbf{P}\Big( \frac{M_t}{b(t)}\geq K \Big)
 = \mathbf{P}\Big( \frac{Y_t+J_t}{b(t)}\geq K \Big)
 \geq \mathbf{P}\Big( \frac{Y_t}{b(t)}\geq K, J_t=0 \Big) = \mathbf{P}\Big( \frac{Y_t}{b(t)}\geq K \Big) \times \mathbf{P}\big(  J_t=0 \big) ,
 \eeqnn 
 which along with \eqref{eqn.343} yields that 
 \beqlb\label{eqn.344}
 \mathbf{P}\Big( \frac{Y_t}{b(t)}\geq K \Big)  \leq \frac{1}{4}. 
 \eeqlb
 On the other hand,  a simple calculation shows that
 \beqnn
 \mathbf{E}\Big[ \frac{Y_t}{b(t)} \Big] = t\cdot \int_{-\infty}^{-b(t)} \frac{|y|}{b(t)}\,\nu(dy)\geq 0 ,\quad 
 \mathrm{Var}\Big(\frac{Y_t}{b(t)} \Big) = \frac{\sigma^2\cdot t}{|b(t)|^2} + t\cdot\int_{-b(t)}^1 \Big|\frac{y}{b(t)}\Big|^2 \,\nu(dy)\leq C_0.
 \eeqnn
 If $ \mathbf{E}\big[ Y_t/b(t)  \big]\leq K $, we see that 
 \beqnn
 t\cdot \int_{-\infty}^{-b(t)} \frac{|y|}{b(t)}\,\nu(dy) \leq K.
 \eeqnn
 Otherwise, if $ \mathbf{E}\big[ Y_t/b(t)  \big]>K $, by  Chebyshev's inequality we have 
 \beqnn
 \mathbf{P}\Big( \frac{Y_t}{b(t)}\geq K \Big) 
 = 1- \mathbf{P}\Big( \frac{Y_t}{b(t)}< K \Big) 
 \ar=\ar 1- \mathbf{P}\Big( \frac{Y_t}{b(t)}-\mathbf{E}\Big[ \frac{Y_t}{b(t)} \Big] < K-\mathbf{E}\Big[ \frac{Y_t}{b(t)} \Big]  \Big) \cr
 \ar\geq\ar  1- \mathbf{P}\bigg( \Big|\frac{Y_t}{b(t)}-\mathbf{E}\Big[ \frac{Y_t}{b(t)} \Big]\Big| > \mathbf{E}\Big[ \frac{Y_t}{b(t)} \Big] -K  \bigg)\cr
 \ar\geq\ar 1- \frac{ \mathrm{Var}\big( Y_t/b(t) \big)}{\big| \mathbf{E}\big[ Y_t/b(t)  \big]-K \big|^2}\cr
 \ar\geq \ar  1- \frac{ C_0}{\big| \mathbf{E}\big[ Y_t/b(t)  \big]-K \big|^2},
 \eeqnn
 which along with \eqref{eqn.344} induces that for all $t\geq t_0$, 
 \beqnn
 \bigg|  \mathbf{E}\Big[ \frac{Y_t}{b(t)} \Big]-K \bigg|^2 \leq  4C_0 
 \quad \mbox{and then }\quad 
 0\leq   \mathbf{E}\Big[ \frac{Y_t}{b(t)} \Big] =t\cdot \int_{-\infty}^{-b(t)} \frac{|y|}{b(t)}\,\nu(dy) \leq  2\sqrt{C_0} +K.
 \eeqnn
 Here we have proved the second inequality in \eqref{eqn.340}. 
 \qed 
 
 \textbf{\textit{Proof of Lemma~\ref{Lemma.302}(2).}} Note that $M$  has  no positive jumps larger than $1$, the L\'evy-Khintchine formula gives 
 \beqnn
 \mathbf{E}\Big[ \exp\Big\{\frac{M_t}{b(t)}\Big\} \Big] 
 = \exp\bigg\{ \frac{\sigma^2}{2} \cdot  \frac{ t}{|b(t)|^2} + t\cdot \int_{-\infty}^1 \Big( \exp\Big\{\frac{y}{b(t)} \Big\} -1- \frac{y}{b(t)}  \Big) \nu(dy) \bigg\}.
 \eeqnn
 By using the inequality $0\leq e^{z}-1-z \leq  |z|^2\wedge |z|$ for all $z\leq 1$, then   \eqref{eqn.336} and Proposition~\ref{Prop.306},  
 \beqnn
 \sup_{t\geq t_0} \mathbf{E}\Big[ \exp\Big\{\frac{M_t}{b(t)}\Big\} \Big]  
 \leq \sup_{t\geq t_0}  \exp\bigg\{ \frac{\sigma^2}{2} \cdot  \frac{ t}{|b(t)|^2} + t\cdot \int_{-\infty}^1  \Big|\frac{y}{b(t)}  \Big|^2 \wedge \Big| \frac{y}{b(t)}  \Big|\, \nu(dy)\bigg\} <\infty.
 \eeqnn 
 \qed 
 
 \textbf{\textit{Proof of Lemma~\ref{Lemma.302}(3).}}    
 Since $M$  has zero mean and no jumps larger than $1$, we have 
 \beqnn
 \mathbf{E}\big[ \exp\big\{ \theta M_1\big\} \big] <\infty ,\quad \forall \, \theta \geq 0
 \quad \mbox{and also}\quad 
 \lim_{\theta \to 0+} \frac{1}{\theta}\cdot \log \mathbf{E}\big[ \exp\big\{ \theta M_1\big\} \big] =0,
 \eeqnn 
 which allows us to choose a constant $c_\eta>0$ sufficiently small such that
 \beqnn
 \mathbf E\Big[  \exp\big\{ 2c_\eta\cdot M_1 \big\} \Big]
 \leq
 \exp\big\{\eta c_\eta\big\}.
 \eeqnn
 By the  Markov inequality and the i.i.d. increments of $M$, we have
 for all $t>0$ and $y\geq \eta t$,
 \beqnn
 \mathbf P\big( M_t>y \big)
 \leq
 e^{-2c_\eta y} \cdot \mathbf E\big[ \exp\big\{ 2c_\eta M_t \big\} \big] 
 \ar\leq\ar
 e^{   c_\eta \eta t  -2c_\eta y  }
 \leq
 e^{- c_\eta y }.
 \eeqnn 
 \qed 
 
 \textbf{\textit{Proof of Lemma~\ref{Lemma.302}(4).}}    
 By the inequality $|e^{z}-1-z| \leq e^{z\vee 0}\cdot (|z|^2\wedge |z|)$ for all $z\in \mathbb{R}$,  there exists a constant $C>0$ such that for all $\lambda \geq 1$,
 \beqnn
 \log\mathbf E\Big[  \exp\big\{\lambda M_1 \big\}  \Big]
 \ar=\ar
 \frac{\sigma^2 \lambda^2}{2} + \int_{(-\infty,1]}  \big( e^{\lambda y}-1-\lambda y \big)\, \nu(dy) 
 \leq 
 C\cdot \lambda^2\cdot e^\lambda.
 \eeqnn
 Choose $\eta\in(0,\frac{1-\varepsilon}{2})$ and set
 $ \lambda=(1-\eta)\log(y/t)$. 
 Note that $\lambda \geq 1$ for all $y>T\cdot e^{1/(1-\eta)}$, we have as $y\to\infty$, 
 \beqnn
 \frac{ \mathbf P(M_t>y) }{t} \ar\leq\ar
 \frac{1}{t} \cdot \exp\big\{ -\lambda y + t\cdot \log\mathbf E\big[ \exp\big\{ \lambda M_1 \big\} \big] \big\} \cr
 \ar\leq\ar \exp\Big\{ -\log t - (1-\eta)y\log\frac{y}{t}  + Ct^\eta y^{1-\eta} \Big( \log\frac{y}{t} \Big)^2 \Big\},
 \eeqnn
 which can be bounded by $ e^{-\frac{1+\varepsilon}{2} \cdot y\log y}=o(e^{- \varepsilon \cdot y\log y})$ uniformly in $t\in(0,T]$, since $ \sup_{t\in(0,T]} t^\eta\big(1+|\log t|\big)^2<\infty$ and $\sup_{t\in(0,T]} \big( (1-\eta)y-1  \big)\log t \leq  C_T\cdot y$ for all large $y$. 
 \qed

 \subsection{Proof of Lemma~\ref{Lemma.301}}
 
 Recall our assumption that $\overline{\nu}(1)=1$. 
 Consider a sequence $\{Y_j\}_{j\geq 1}$ of i.i.d. positive random variables with common distribution $ \nu_1(dy) := \mathbf 1_{\{y>1\}}\cdot \nu(dy)$. 
 In view of our assumption, we have 
 \beqnn
 m:=\mathbf{E}[Y_1] <\int_1^\infty y^\alpha \nu(dy)<\infty 
 \quad\mbox{and}\quad  
 \mathbf{E}\big[ |Y_1|^\alpha \big]<\infty. 
 \eeqnn 
 Let $S=\{S_n\}_{n\geq0}$ and $\widetilde{S}=\{\widetilde{S}_n\}_{n\geq0}$ be two random walks defined by $S_0=\widetilde{S}_0=0$,
 \beqnn
 S_n:= \sum_{i=1}^n Y_i
 \quad \mbox{and}\quad 
 \widetilde{S}_n:= \sum_{i=1}^n \big(Y_i -m\big).
 \eeqnn
 Let $N=\{ N_t:t\geq 0 \}$ be a Poisson process with rate $1$. The process $Z$ has the following realization:
 \beqlb\label{eqn.3001}
 Z_t\overset{\rm d} = S_{N_t}-m\cdot t = \sum_{j=1}^{N_t} Y_j - m\cdot t  .
 \eeqlb

 \begin{proposition}\label{Prop.3041}
 	There exist two constants $C,t_0>0$ such that for all $t\geq t_0$ and $K\geq 1$,
 	\beqnn
 	\sup_{0\leq s\leq t}\mathbf{P}\big(  |N_s-   s | \geq Kb(t)\big) + 	\sup_{0< s\leq t} \mathbf{E}\Big[\frac{N_s}{s}\cdot \mathbf{1}_{\{ |N_s- s| \geq Kb(t)\}}\Big]\leq \frac{C}{K} .
 	\eeqnn
 \end{proposition}
 \proof  By Proposition~\ref{Prop.306} and  Chebyshev's inequality, we first have 
 \beqnn
 \sup_{0\leq s\leq t}\mathbf{P}\big(  |N_s-   s | \geq Kb(t)\big) \leq \sup_{0\leq s\leq t} \frac{\mathbf{E}\big[ |N_s-   s |^2\big]}{ |Kb(t)|^2} = \frac{ t}{|Kb(t)|^2} \leq \frac{C}{K^2},
 \eeqnn
 for some constant $C>0$ independent of $t$. 
 Moreover, by the inequality $N_s\leq s+ |N_s-s| $ and  the Markov inequality, we also have for $t \geq t_0$,
 \beqnn
 \sup_{0< s\leq t} \mathbf{E}\Big[\frac{N_s}{s}\cdot \mathbf{1}_{\{ |N_s- s| \geq Kb(t)\}}\Big] 
 \ar\leq\ar  \sup_{0< s\leq t} \mathbf{P}\big(  |N_s- s| \geq Kb(t) \big) \cr
 \ar\ar + \sup_{0<s\leq t} \mathbf{E}\Big[\frac{|N_s-s|}{s} \cdot \mathbf{1}_{\{ |N_s- s| \geq Kb(t)\}}\Big]  \cr
 \ar\leq\ar  \frac{C}{K^2} + \sup_{0< s\leq t} \frac{\mathbf{E}\big[ |N_s-   s |^2\big]}{ s\cdot Kb(t) }\cr
 \ar =\ar  \frac{C}{K^2} +  \frac{1}{Kb(t) },
 \eeqnn
 and the desired upper bound holds since the function $b$ is eventually non-decreasing. 
 \qed 
 
 \begin{proposition}\label{Prop.304}
 	For any $\delta_1\geq \delta_0>0$, $K\geq 0$ and $n\geq1$, we have as $x\to\infty$,
 	\beqlb\label{eqn.333}
 	\sup_{|y|\leq Kb(x)}\sup_{\delta\in[\delta_0,\delta_1]} \bigg|\frac{  \mathbf{P}\big(  S_n \in x-y+\varDelta_\delta \big) }{n\cdot  \nu\big(x+\varDelta_\delta\big)} -1\bigg| \to 0.
 	\eeqlb
 	Moreover,  for any $\rho >0$, there exist two constants $C,x_0>0$ such that for all $n\geq 1$, $x\geq x_0$ and $\delta \in[\delta_0,\delta_1]$,
 	\beqlb\label{eqn.334}
 	\sup_{|y|\leq Kb(x)} \mathbf{P}\Big(  S_n   \in x-y+\varDelta_\delta \Big) \leq C \cdot (1+\rho)^n \cdot  \nu\big(x+\varDelta_\delta\big).
 	\eeqlb
 \end{proposition}
 \proof 
 Here we just need to prove these results with $K=0$.
 For the case of $K>0$, they follow directly from Proposition~\ref{Prop.205}.  
 For any $\delta>0$, recall that $x^{\alpha}\cdot \nu_1\big(x+\varDelta_\delta\big)$ belongs to either $\mathcal{OR}$ or $\mathcal{S}d$.  Lemma 6.1 and 6.2 in \cite{DenisovDiekerShneer2008} shows that  $\{  n^{1/\alpha}\}_{n\geq 1}$ is a truncation sequence for $\nu_1(dy)$.  
 Additionally, since $\mathbf{E}[Y_1] <\infty$, then $\{ n \}_{n\geq 1}$ is a natural-scale sequence. 
 For any $\delta >0$, by Corollary~2 and Proposition~4 in \cite{AsmussenFossKorshunov2003} we have as $x\to\infty$, 
 \beqlb\label{eqn.3331}
 \bigg|\frac{  \mathbf{P}\big(  S_n \in x+\varDelta_\delta \big) }{n\cdot  \nu\big(x+\varDelta_\delta\big)} -1\bigg| \to 0 
 \eeqlb
 and for some constant $x_0>0$,
 \beqlb \label{eqn.3332}
 \mathbf{P}\Big(  S_n   \in x +\varDelta_\delta \Big) \leq C \cdot (1+\rho)^n \cdot  \nu\big(x+\varDelta_\delta\big),
 \eeqlb
 uniformly in $n\geq 1$, $x\geq x_0$.  
 For any $\epsilon \in(0,1)$, recall the finite partition $\delta_0= r_0<r_1<\cdots r_{k_\epsilon}= \delta_1$ defined in the proof of Proposition~\ref{Prop.205}, whose mesh size is at most $\epsilon$.   
 For any $k=0,\cdots, k_\epsilon-1$ and any $\delta\in [r_k,r_{k+1}]$, we have 
 \beqnn
 \frac{  \mathbf{P}\big( S_n\in x+\varDelta_{r_{k}} \big) }{n\cdot  \nu\big(x+\varDelta_{r_{k+1}}\big)}
 \leq \frac{  \mathbf{P}\big( S_n \in x+\varDelta_\delta \big) }{n\cdot  \nu\big(x+\varDelta_\delta\big)}
 \leq 	 \frac{  \mathbf{P}\big( S_n \in x+\varDelta_{r_{k+1}}\big) }{n\cdot  \nu\big(x+\varDelta_{r_{k}}\big)}. 
 \eeqnn
 which follows that 
 \beqnn
 \inf_{k=0,\cdots, k_{\epsilon}-1}    \frac{  \mathbf{P}\big( S_n \in x+\varDelta_{r_{k}} \big) }{n\cdot  \nu\big(x+\varDelta_{r_{k+1}}\big)}
 \ar\leq\ar \inf_{\delta\in[\delta_0,\delta_1]} \frac{  \mathbf{P}\big( S_n\in x+\varDelta_\delta \big) }{n\cdot  \nu\big(x+\varDelta_\delta\big)}\cr
 \ar\leq\ar \sup_{\delta\in[\delta_0,\delta_1]} \frac{  \mathbf{P}\big( S_n\in x+\varDelta_\delta \big) }{n\cdot  \nu\big(x+\varDelta_\delta\big)}
 \leq  \sup_{k=0,\cdots, k_{\epsilon}-1}  \frac{  \mathbf{P}\big( S_n \in x+\varDelta_{r_{k+1}}\big) }{n\cdot  \nu\big(x+\varDelta_{r_{k}}\big)}. 
 \eeqnn
 Since $k_\epsilon<\infty$, the uniform upper bound \eqref{eqn.334} with $K=0$ follows directly from \eqref{eqn.3332}  and Proposition~\ref{Prop.207}. 
 Moreover, armed with \eqref{eqn.3331} and  Proposition~\ref{Prop.207} we pass the first and the last terms to their corresponding limits,
 \beqnn
 \lim_{x\to\infty}\inf_{k=0,\cdots, k_{\epsilon}-1}    \frac{  \mathbf{P}\big( S_n \in x+\varDelta_{r_{k}} \big) }{n\cdot  \nu\big(x+\varDelta_{r_{k+1}}\big)}
 \ar=\ar \inf_{k=0,\cdots, k_{\epsilon}-1}   \frac{r_k}{r_{k+1}}  \geq \frac{\delta_0}{\delta_0+\epsilon}
 \eeqnn
 and 
 \beqnn
 \lim_{x\to\infty}\sup_{k=0,\cdots, k_{\epsilon}-1}  \frac{  \mathbf{P}\big( S_n \in x+\varDelta_{r_{k+1}}\big) }{n\cdot  \nu\big(x+\varDelta_{{r_{k}}}\big)} = \sup_{k=0,\cdots, k_{\epsilon}-1}   \frac{r_{k+1}}{r_k}  \leq 1+ \frac{\epsilon}{\delta_0}. 
 \eeqnn
 By the squeeze theorem, we have  
 \beqnn
 \frac{\delta_0}{\delta_0+\epsilon}  \ar\leq\ar \liminf_{x\to\infty} \inf_{\delta\in[\delta_0,\delta_1]}\frac{  \mathbf{P}\big( S_n \in x+\varDelta_\delta \big) }{n\cdot  \nu\big(x+\varDelta_\delta\big)}
 \leq  \limsup_{x\to\infty} \sup_{\delta\in[\delta_0,\delta_1]}\frac{  \mathbf{P}\big( S_n \in x+\varDelta_\delta \big) }{n\cdot  \nu\big(x+\varDelta_\delta\big)} \leq 1+ \frac{\epsilon}{\delta_0},
 \eeqnn
 and the desired limit \eqref{eqn.333} follows immediately as $\epsilon \to 0+$. 
 \qed 
 
 \begin{proposition}\label{Prop.305}
 	For any $ \theta>0$, $K\geq 0$ and $\delta_1\geq \delta_0>0$, we have as $n\to\infty$,
 	\beqlb\label{eqn.3261}
 	\sup_{x\geq \theta n} \sup_{|y|\leq Kb(x)}\sup_{\delta\in[\delta_0,\delta_1]}
 	\bigg| \frac{  \mathbf P\big(  \widetilde{S}_n \in x-y+\varDelta_\delta\big)}{n\cdot \nu\big(x+\varDelta_\delta\big)}-1
 	\bigg| \to 0.
 	\eeqlb
 	Moreover, there exists a constant $C>0$ such  that for all $n\geq 1$, $x\geq \theta n$ and $\delta \in [\delta_0,\delta_1]$, 
 	\beqlb\label{eqn.3262}
 	\sup_{|y|\leq Kb(x)} \mathbf P\big(  \widetilde{S}_n \in x-y+\varDelta_\delta\big) \leq C\cdot n\cdot \nu\big(x+\varDelta_\delta\big). 
 	\eeqlb
 	
 \end{proposition}
 \proof Similarly as in the proof of Proposition~\ref{Prop.304}, we just need to prove them with $K=0$. 
 Since $\mathbf{E}\big[ |Y_1-m|^\alpha \big]<\infty$,  Corollary~2.1 in  \cite{DenisovDiekerShneer2008} along with Proposition~\ref{Prop.205} tells that for any $\theta>0$ and  $\delta>0$,
 \beqlb\label{eqn3.2}
 \lim_{n\to\infty}
 \sup_{x\geq \theta n}
 \bigg|
 \frac{ \mathbf P\big(  \widetilde{S}_n \in x+\varDelta_\delta\big) }{ n \cdot \nu\big(x+\varDelta_\delta\big) } -1 \bigg|
 =0.
 \eeqlb
 We now prove that this limit also holds uniformly	 in  $\delta\in[\delta_0,\delta_1]$.
 For  any $\epsilon\in(0,1)$, we recall the partition $\delta_0=r_0<\cdots<r_{k_\epsilon}=\delta_1$, have that for any $k=0,\cdots ,k_\epsilon-1$ and $\delta \in [r_k,r_{k+1}]$, 
 \beqnn
 \frac{ \mathbf P\big(  \widetilde{S}_n \in x+\varDelta_{r_k}\big) }{ n \cdot \nu\big(x+\varDelta_{r_{k+1}}\big) } \leq  \frac{ \mathbf P\big(  \widetilde{S}_n \in x+\varDelta_\delta\big) }{ n \cdot \nu\big(x+\varDelta_\delta\big) } 
 \leq  \frac{ \mathbf P\big(  \widetilde{S}_n \in x+\varDelta_{r_{k+1}}\big) }{ n \cdot \nu\big(x+\varDelta_{r_{k}}\big) }. 
 \eeqnn
 By \eqref{eqn3.2} and  Proposition~\ref{Prop.207}, we obtain
 \beqnn
 \frac{1}{1+\epsilon/\delta_0}
 \ar\leq\ar
 \liminf_{n\to\infty}\inf_{x\geq \theta n}\inf_{\delta\in [\delta_0,\delta_1]}
 \frac{ \mathbf P\big(  \widetilde{S}_n \in x+\varDelta_\delta\big) }{ n \cdot \nu\big(x+\varDelta_\delta\big) } 
 \leq
 \limsup_{n\to\infty}
 \sup_{x\geq \theta n}\sup_{\delta\in [\delta_0,\delta_1]}
 \frac{ \mathbf P\big(  \widetilde{S}_n \in x+\varDelta_\delta\big) }{ n \cdot \nu\big(x+\varDelta_\delta\big) } 
 \leq
 1+\frac{\epsilon}{\delta_0},
 \eeqnn
 and the desired limit \eqref{eqn.3261} with $K=0$ follows as  $\epsilon\to 0+$. 
 For the uniform upper bound \eqref{eqn.3262}, the limit \eqref{eqn.3261} tells that for some large $N$,  
 \beqnn
 \sup_{n\geq N}   \sup_{x\geq \theta n}\sup_{\delta\in [\delta_0,\delta_1]}
 \frac{ \mathbf P\big( \widetilde{S}_n \in x+\varDelta_\delta\big) }{ n \cdot \nu\big(x+\varDelta_\delta\big) } 
 \leq2.
 \eeqnn
 Combining this with the second claim in Proposition~\ref{Prop.304} induces the uniform upper bound \eqref{eqn.3262} immediately.
 \qed  
 
 \textbf{\textit{Proof of Lemma~\ref{Lemma.301}.}} 
 Without loss of generality, we may assume $\theta \in(0,1]$, $\delta_0=1$ and $\delta_1=2$. The general case can be proved in the same way. 
 Here we just need to prove the desired limit \eqref{eqn.301} with $K=0$, i.e., 
 \beqlb\label{eqn.325}
 \lim_{t\to\infty}\sup_{0<s\leq t} \sup_{x\geq  \theta t}   \sup_{\delta\in[1,2]} 
 \bigg| 	\frac{ \mathbf P\big( Z_s\in x+\varDelta_\delta \big)}{s\cdot \nu\big(x+\varDelta_\delta\big)} -1 \bigg| = 0.
 \eeqlb 
 In the case of $K>0$, by the triangle inequality and the assumption $b(x)=o(x)$ as $x\to\infty$,  we have for large $t$, 
 \beqnn
 \lefteqn{\sup_{0<s\leq t} \sup_{x\geq  \theta t} 
 	\sup_{|y|\leq Kb(x)}  \sup_{\delta\in[1,2]} 
 	\bigg| 	\frac{ \mathbf P\big( Z_s\in x-y+\varDelta_\delta \big)}{s\cdot \nu\big(x+\varDelta_\delta\big)} -1 \bigg| }\ar\ar\cr
 \ar\leq \ar  \sup_{0<s\leq t} \sup_{z\geq  \theta t/2}  \sup_{\delta\in[1,2]} 
 \bigg| 	\frac{ \mathbf P\big( Z_s\in z+\varDelta_\delta \big)}{s\cdot \nu\big(z+\varDelta_\delta\big)} -1 \bigg| \cdot  \sup_{0<s\leq t} \sup_{x\geq  \theta t}   \sup_{|y|\leq Kb(x)} \sup_{\delta\in[1,2]}  \frac{\nu\big(x-y+\varDelta_\delta\big)}{\nu\big(x+\varDelta_\delta\big)} \cr
 \ar\ar +  \sup_{0<s\leq t} \sup_{x\geq  \theta t}  \sup_{|y|\leq Kb(x)}  \sup_{\delta\in[1,2]}   \bigg|   \frac{\nu\big(x-y+\varDelta_\delta\big)}{\nu\big(x+\varDelta_\delta\big)} -1\bigg|
 \eeqnn
 which goes to $0$ as $t\to\infty$ because of \eqref{eqn.325} and Proposition~\ref{Prop.205}. 
 
 We start to prove the limit \eqref{eqn.325}. Recall the random walk $\{\widetilde{S}_n\}_{n\geq 0}$ defined in the proof of Proposition~\ref{Prop.305}. 
 By \eqref{eqn.3001}, we have 
 \beqnn
 \mathbf{P}\big( Z_s\in x+\varDelta_\delta  \big) 
 \ar=\ar    \mathbf{P}\bigg( \sum_{j=1}^{N_s}Y_j-m\cdot s\in x+\varDelta_\delta  \bigg)    \cr 
 \ar=\ar \sum_{n=0}^{\infty}  \frac{s^n
 }{n!}e^{- s} \cdot	\mathbf{P}\big( \widetilde{S}_n \in x-m\cdot(n -s)+\varDelta_\delta\big),
 \eeqnn
 which, for any $R> 0$, can be decomposed into the following four terms
 \beqlb
 I_1(R,t,s,x,\delta) \ar:=\ar \sum_{n\geq 0,|n- s|\leq Rb(t)} \frac{s^n}{n!} e^{- s} \cdot \mathbf{P}\big( \widetilde{S}_n \in x-m\cdot(n -s)+\varDelta_\delta\big), \label{eqn.327}\\
 I_2(R,t,s,x,\delta) \ar:=\ar \sum_{n\geq 0, n-s<-Rb(t)}  \frac{s^n}{n!}e^{- s} \cdot  \mathbf{P}\big( \widetilde{S}_n \in x-m\cdot(n -s)+\varDelta_\delta\big), \label{eqn.328} \\
 I_3(R,t,s,x,\delta) \ar:=\ar \sum_{ Rb(t)< n- s	<
 	\frac{x}{2m}} \frac{s^n}{n!}e^{-s}  \cdot	 \mathbf{P}\big( \widetilde{S}_n \in x-m\cdot(n -s)+\varDelta_\delta\big), \label{eqn.329} \\
 I_4(R,t,s,x,\delta) \ar:=\ar \sum_{ n- s\geq \frac{x}{2m}} \frac{s^n}{n!}e^{-s}  \cdot	\mathbf{P}\big( \widetilde{S}_n \in x-m\cdot(n -s)+\varDelta_\delta\big).  \label{eqn.330}
 \eeqlb
 It is easy to see that \eqref{eqn.325} follows if we can prove that 
 \beqlb\label{eqn.331}
 \lim_{R\to\infty}\limsup_{t\to\infty}  \sup_{0<s\leq t} \sup_{x\geq \theta t}   \sup_{\delta\in[1,2]} \bigg|	\frac{  I_1(R,t,s,x,\delta)}{s\cdot \nu\big(x+\varDelta_\delta\big)} -1 \bigg|   =0. 
 \eeqlb
 and 
 \beqlb\label{eqn.332}
 \lim_{R\to\infty}\limsup_{t\to\infty}  \sup_{0<s\leq t} \sup_{x\geq \theta t}   \sup_{\delta\in[1,2]} \frac{  I_i(R,t,s,x,\delta)}{s\cdot \nu\big(x+\varDelta_\delta\big)}    =0, \quad i=2,3,4.
 \eeqlb
 We now prove these four limits one-by-one in the following four parts: 
 \begin{enumerate}
 	\item[$\bullet$]  
 	By \eqref{eqn.327} and the fact that $\mathbf{E}[N_s]= \sum_{n=0}^\infty \frac{s^n}{n!} e^{- s} \cdot n =s$, 
 	\beqnn
 	\bigg|\frac{  I_1(R,t,s,x,\delta)}{s\cdot \nu\big(x+\varDelta_\delta\big)} -1 \bigg|
 	\ar\leq \ar \sum_{n\geq 0, |n- s|> Rb(t)} \frac{s^n}{n!} e^{- s} \cdot \frac{n}{s}\cr
 	\ar\ar + \sum_{n\geq0,|n- s|\leq Rb(t)} \frac{s^n}{n!} e^{- s} \cdot \bigg| \frac{\mathbf{P}\big( \widetilde{S}_n \in x-m\cdot(n -s)+\varDelta_\delta\big)}{s\cdot \nu\big(x+\varDelta_\delta\big)}-\frac{n}{s} \bigg| , 
 	\eeqnn
 	which along with  the triangle inequality induces that  
 	\beqlb\label{eqn.3002}
 	\sup_{0<s\leq t} \sup_{x\geq \theta t}   \sup_{\delta\in[1,2]}  \bigg| 	\frac{  I_1(R,t,s,x,\delta)}{s\cdot \nu\big(x+\varDelta_\delta\big)} -1 \bigg| 
 	\leq    \sup_{0< s\leq t} \mathbf{E}\Big[\frac{N_s}{s}\cdot \mathbf{1}_{\{ |N_s- s| \geq Rb(t)\}}\Big] +  \varepsilon(R,t)
 	\eeqlb
 	with 
 	\beqnn
 	\varepsilon(R,t):= \sup_{0<s\leq t}\sup_{n>0, |n- s|\leq Rb(t)}  \sup_{x\geq \theta t}   \sup_{\delta\in[1,2]}  \bigg| \frac{\mathbf{P}\big( \widetilde{S}_n \in x-m\cdot(n -s)+\varDelta_\delta\big)}{n\cdot \nu\big(x+\varDelta_\delta\big)}-1\bigg| .
 	\eeqnn
 	By Proposition~\ref{Prop.3041}, the first term on the right-hand side of \eqref{eqn.3002} goes to $0$ as $t\to\infty$ and then $R\to\infty$.   
 	Moreover, since $b(t)=o(t)$ as $t\to\infty$, there exist two constants $C, t_0>0$ such that for all $t\geq t_0$, $0<s\leq t$, $ |n- s|\leq Rb(t)$ and $x\geq \theta t$,
 	\beqnn
 	n\leq 2t <\frac{2}{\theta } x \quad \mbox{and} \quad \sup_{0<s\leq t} m\cdot |n-s| \leq C\cdot b(x)  ,
 	\eeqnn 
 	which along with Proposition~\ref{Prop.304} (for small $n$) and \ref{Prop.305} (for large $n$) induces that 
 	\beqnn
 	\lim_{R\to\infty}\limsup_{t\to\infty}\varepsilon(R,t) =0,
 	\eeqnn
 	and then the limit \eqref{eqn.331} follows.
 	
 	\item[$\bullet$] We now prove \eqref{eqn.332} with $i=2$. By \eqref{eqn.328}, 
 	\beqnn
 	\lefteqn{\sup_{0<s\leq t} \sup_{x\geq \theta t}   \sup_{\delta\in[1,2]} \frac{  I_2(R,t,s,x,\delta)}{s\cdot \nu\big(x+\varDelta_\delta\big)}  }\ar\ar\cr
 	\ar=\ar \sup_{0<s\leq t} \sup_{x\geq \theta t}   \sup_{\delta\in[1,2]}  \sum_{ 1\leq n<s-Rb(t)}   \frac{s^n}{n!}e^{- s} \cdot	\frac{ \mathbf{P}\big( \widetilde{S}_n \in x-m\cdot(n -s)+\varDelta_\delta\big)}{s\cdot \nu\big(x+\varDelta_\delta\big)} \cr
 	\ar\leq\ar \sup_{0<s\leq t} \sup_{x\geq \theta t}   \sup_{\delta\in[1,2]}  \sup_{   1\leq n<s-Rb(t)} \frac{ \mathbf{P}\big( \widetilde{S}_n \in x-m\cdot(n -s)+\varDelta_\delta\big) }{n\cdot \nu\big(x+\varDelta_\delta\big)} \cdot \mathbf{P}\big(N_s\leq s-Rb(t) \big). 
 	\eeqnn
 	By Proposition~\ref{Prop.3041}, we have as $t\to\infty$ and then $R\to\infty$, 
 	\beqnn
 	\sup_{0<s\leq t} \mathbf{P}\big(N_s\leq s-Rb(t) \big) 
 	\leq   \sup_{0<s\leq t} \mathbf{P}\big(|N_s-s|\geq Rb(t) \big)  \to 0. 
 	\eeqnn
 	Additionally, in view of \eqref{eqn.2021}, there exist three constants $C,\varepsilon ,t_0>0$ such that  for all $t\geq t_0$, $0<s\leq t$, $n-s< -Rb(s)$ and $x\geq \theta t$, 
 	\beqnn
 	x-m\cdot(n -s) \geq  \varepsilon n
 	\quad \mbox{and}\quad 
 	\sup_{\delta\in[1,2]} \frac{\nu\big(x-m\cdot(n -s) +\varDelta_\delta\big)}{\nu\big(x+\varDelta_\delta\big)} \leq C, 
 	\eeqnn
 	which along with the second claim in Proposition~\ref{Prop.305} induces that 
 	\beqnn
 	\lefteqn{\sup_{0<s\leq t} \sup_{x\geq \theta t}   \sup_{\delta\in[1,2]}  \sup_{  n<s-Rb(t)} \frac{ \mathbf{P}\big( \widetilde{S}_n \in x-m\cdot(n -s)+\varDelta_\delta\big) }{n\cdot \nu\big(x+\varDelta_\delta\big)} }\ar\ar\cr
 	\ar\leq\ar C\cdot \sup_{0<s\leq t} \sup_{x\geq \theta t}   \sup_{\delta\in[1,2]}  \sup_{ n<s-Rb(t)} \frac{ \mathbf{P}\big( \widetilde{S}_n \in x-m\cdot(n -s)+\varDelta_\delta\big) }{n\cdot \nu\big(x-m\cdot(n -s)+\varDelta_\delta\big)}\cr
 	\ar\leq\ar C\cdot  \sup_{  n \geq 1}\sup_{x\geq \varepsilon n}   \sup_{\delta\in[1,2]}  \frac{ \mathbf{P}\big( \widetilde{S}_n \in  x+\varDelta_\delta\big) }{n\cdot \nu\big(x+\varDelta_\delta\big)} \leq C, 
 	\eeqnn
 	uniformly in $t\geq t_0$.  
 	Combining these results together, we can get \eqref{eqn.332} with $i=2$ immediately.

 	\item[$\bullet$] We now prove \eqref{eqn.332} with $i=3$. By \eqref{eqn.329}, 
 	\beqlb\label{eqn.335}
 	\lefteqn{\sup_{0<s\leq t} \sup_{x\geq \theta t}   \sup_{\delta\in[1,2]} \frac{  I_3(R,t,s,x,\delta)}{s\cdot \nu\big(x+\varDelta_\delta\big)}  }\ar\ar\cr
 	\ar=\ar \sup_{0<s\leq t} \sup_{x\geq \theta t}   \sup_{\delta\in[1,2]}  \sum_{ Rb(t)< n- s	< \frac{x}{2m}}  \frac{s^n}{n!}e^{- s} \cdot	\frac{ \mathbf{P}\big( \widetilde{S}_n \in x-m\cdot(n -s)+\varDelta_\delta\big) }{s\cdot \nu\big(x+\varDelta_\delta\big)} .
 	\eeqlb
 	There exist some constants $t_0, C>0$ such that for all $t\geq t_0$, $0<s\leq t$, $x\geq \theta t$ and  $Rb(t)\leq n-s<\frac{x}{2m}$,
 	\beqnn
 	\frac{x}{2}\leq x-m\cdot (n-s) \leq x
 	\quad \mbox{and}\quad 
 	n\leq C\cdot x.
 	\eeqnn
 	which along with the second claim in Proposition~\ref{Prop.305} induces that 
 	\beqnn
 	\frac{ \mathbf{P}\big( \widetilde{S}_n \in x-m\cdot(n -s)+\varDelta_\delta\big) }{s\cdot \nu\big(x+\varDelta_\delta\big)} 
 	\leq C\cdot  \frac{n}{s}  \cdot \frac{\nu\big(x-m\cdot(n -s)+\varDelta_\delta\big)}{\nu\big(x+\varDelta_\delta\big)}  .
 	\eeqnn
 	By Proposition~\ref{Prop.203}, there exists some constant $C>0$ such that for large $t$,
 	\beqnn
 	\sup_{x\geq \theta t}   \sup_{\delta\in[1,2]}  \frac{ \mathbf{P}\big( \widetilde{S}_n \in x-m\cdot(n -s)+\varDelta_\delta\big)}{s\cdot \nu\big(x+\varDelta_\delta\big)} 
 	\leq C \cdot  \frac{n}{s} \cdot \exp \Big\{   \frac{m(n-s)}{b(x)} \Big\},
 	\eeqnn
 	uniformly in $0<s\leq t$ and $Rb(t)\leq n-s<\frac{x}{2m}$. 
 	Plugging this back into \eqref{eqn.335}, then using the non-decreasing property of $b$ and the Markov inequality in the third and fourth inequality respectively,
 	\beqlb\label{eqn.337}
 	\sup_{0<s\leq t} \sup_{x\geq \theta t}   \sup_{\delta\in[1,2]} \frac{  I_3(R,t,s,x,\delta)}{s\cdot \nu\big(x+\varDelta_\delta\big)}
 	\ar\leq\ar C \sup_{0<s\leq t} \sup_{x\geq \theta t}   \sum_{ Rb(t)\leq n- s	< \frac{x}{2m}}  \frac{s^{n-1}}{(n-1)!}e^{- s} \cdot \exp \Big\{  \frac{m(n-s)}{b(x)} \Big\}\cr
 	\ar\leq\ar  C \sup_{0<s\leq t}    \sum_{n\geq s+ Rb(t)-1}  \frac{s^n}{n!}e^{- s} \cdot \exp \Big\{  \frac{m(n-s+1)}{b(\theta t)} \Big\} \cr
 	\ar=\ar C \sup_{0<s\leq t} \mathbf{E}\Big[ \exp \Big\{  \frac{m(N_s-s+1)}{b(\theta t)} \Big\}\cdot \mathbf{1}_{\{ N_s-s\geq   Rb(t)-1 \}}  \Big]\cr
 	\ar\leq\ar C\cdot \sup_{0<s\leq t}\mathbf{E}\Big[ \exp \Big\{  \frac{m(N_s-s+1)}{b(\theta t)} \Big\}\cdot \mathbf{1}_{\{ N_s-s+1\geq   Rb(\theta t) \}}  \Big]\cr
 	\ar\leq\ar C\cdot e^{-R} \cdot\sup_{0<s\leq t} \mathbf{E}\Big[ \exp \Big\{ (m+1) \frac{(N_s-s+1)}{b(\theta t)}  \Big\}  \Big].
 	\eeqlb
 	By using the Poisson moment generating function and then the inequality $e^z-1-z\leq z^2$ for all $z\in[0,1]$, we have for large $t $, 
 	\beqnn
 	\sup_{0<s\leq t} \mathbf{E}\Big[ \exp \Big\{ (m+1) \frac{(N_s-s+1)}{b(\theta t)}  \Big\}  \Big] 
 	\ar=\ar \sup_{0<s\leq t} \exp\Big\{\frac{m+1}{b(\theta t)} + s\cdot \Big(e^{\frac{m+1}{b(\theta t)}}-1-\frac{m+1}{b(\theta t)} \Big)\Big\} \cr
 	\ar\leq\ar \exp\Big\{ \frac{m+1}{b(\theta t)} +  \frac{(m+1)^2t}{\big(b(\theta t)\big)^2}  \Big\} .
 	\eeqnn 
 	Plugging this into the right-hand side of the last inequality in \eqref{eqn.337} and then using   \eqref{eqn.336},  
 	\beqnn
 	\lim_{t\to\infty}\sup_{0<s\leq t} \sup_{x\geq \theta t}   \sup_{\delta\in[1,2]} \frac{  I_3(R,t,s,x,\delta)}{s\cdot \nu\big(x+\varDelta_\delta\big)} 
 	\leq C   \lim_{t\to\infty} \exp\Big\{ -R+ \frac{m+1}{b(\theta t)} +  \frac{(m+1)^2t}{\big(b(\theta t)\big)^2} \Big) \Big\} \leq C\cdot e^{-R},
 	\eeqnn
 	which vanishes as $R\to\infty$. Here we have proved \eqref{eqn.332} with $i=3$.
 	
 	\item[$\bullet$] For  \eqref{eqn.332} with $i=4$, we first prove it holds uniformly in $1/2\leq s\leq t$. By \eqref{eqn.330},   
 	\beqlb\label{eqn.338}
 	I_4(R,t,s,x,\delta) \leq \sum_{ n- s\geq \frac{x}{2m}} \frac{s^n}{n!}e^{-s} = \mathbf{P}\Big( N_s-s  \geq \frac{x}{2m}\Big) . 
 	\eeqlb
 	The Chernoff bound for Poisson random variables gives that for all $x,s>0$, 
 	\beqnn
 	\mathbf{P}\Big( N_s-s  \geq \frac{x}{2m}\Big)  
 	\leq   \exp\Big\{ -s\Big[\Big(1+\frac{x/s}{2m}\Big) \cdot\log \Big(1+\frac{x/s}{2m}\Big) - \frac{x/s}{2m}\Big]  \Big\} .
 	\eeqnn 
 	Note that $(1+1/\theta)\log (1+\theta) -1> 0$, we have $ (1+z)\log(1+z) -z\geq \varepsilon z$ for any $0<\varepsilon<(1+1/\theta)\log (1+\theta) -1$ and $z\geq \theta $. 
 	Thus, for all $0<s\leq t$ and $x\geq \theta t+2m$, 
 	\beqnn
 	\mathbf{P}\Big( N_s-s  \geq \frac{x}{2m}\Big)   \leq   \exp\Big\{  -   \frac{\varepsilon}{2m}\cdot x \Big\}.
 	\eeqnn 
 	Taking this back into \eqref{eqn.338} and then using 	Proposition~\ref{Prop.202}, 
 	\beqnn
 	\limsup_{t\to\infty}\sup_{1/2\leq s\leq t} \sup_{x\geq \theta t}   \sup_{\delta\in[1,2]} \frac{  I_4(R,t,s,x,\delta)}{s\cdot \nu\big(x+\varDelta_\delta\big)} 
 	\ar\leq\ar    \limsup_{t\to\infty}  \sup_{1/2\leq s\leq t} \sup_{x\geq \theta t}   \sup_{\delta\in[1,2]} \frac{\exp\big\{ - \frac{\varepsilon}{2m}\cdot x\big\}}{s\cdot \nu\big(x+\varDelta_\delta\big)}   \cr
 	\ar \leq\ar   \lim_{t\to\infty}  \sup_{x\geq \theta t}      2
 \cdot     \exp\Big\{  -   \frac{\varepsilon}{2m}\cdot x +C_0\cdot x^{1-1/\alpha}    \Big\}  =0. 
 	\eeqnn
 	It remains to consider the case of $0<s <1/2$. 
 	By \eqref{eqn.330}, 
 	\beqnn
 	\frac{  I_4(R,t,s,x,\delta)}{s\cdot \nu\big(x+\varDelta_\delta\big)}   
 	\ar=\ar   \sum_{ n- s\geq \frac{x}{2m}} \frac{s^n}{n!}e^{-s}  \cdot	\frac{\mathbf{P}\big( S_n \in  x+ms +\varDelta_\delta\big)}{s\cdot \nu\big(x+\varDelta_\delta\big)} \cr
 	\ar=\ar   \sum_{ n- s\geq \frac{x}{2m}} \frac{s^{n-1}}{(n-1)!}e^{-s}  \cdot	\frac{\mathbf{P}\big( S_n  \in  x+ms +\varDelta_\delta\big)}{n\cdot \nu\big(x+\varDelta_\delta\big)} .
 	\eeqnn
 	By Proposition~\ref{Prop.304} with $K=1$ and $\rho=1/2$, there exists a constant $C>0$ such that for all large $t$,
 	\beqnn
 	\sup_{0<s<1/2} \sup_{x\geq \theta t}   \sup_{\delta\in[1,2]} \frac{  I_4(R,t,s,x,\delta)}{s\cdot \nu\big(x+\varDelta_\delta\big)}  
 	\ar\leq \ar \sup_{0<s<1/2}   
 	\sum_{ n- s\geq \frac{\theta t}{2m}} \frac{C\cdot (3s/2)^{n-1}}{(n-1)!}e^{-s} 
 	\leq    \sum_{ n\geq \frac{\theta t}{2m}} \frac{C\cdot (3/4)^{n-1}}{(n-1)!} ,
 	\eeqnn
 	which goes to $0$ as $t\to\infty$. Here we have proved \eqref{eqn.332} with $i=4$. 
 \end{enumerate} 
 
 Finally, we prove \eqref{eqn.3011}.
 By \eqref{eqn.2021}, we have for some $x_0>0$,
 \beqlb  \label{eqn.312}
 \sup_{x\geq x_0} \sup_{y<- b(x)} \sup_{\delta\in[1,2]}\frac{  \nu\big(x-y+\varDelta_\delta\big)}{  \nu\big(x+\varDelta_\delta\big)} <\infty.
 \eeqlb
 Moreover,  by the preceding result \eqref{eqn.325}, we have as $t\to\infty$, 
 \beqnn
 \sup_{0<s\leq t} \sup_{x\geq \theta t} \sup_{y<-b(x)} \sup_{\delta\in[1,2]}
 \bigg| 	\frac{ \mathbf P\big( Z_s\in x-y+\varDelta_\delta \big)}{s\cdot \nu\big(x-y+\varDelta_\delta\big)} -1 \bigg| 
 \leq \sup_{0<s\leq t} \sup_{x\geq \theta t}   \sup_{\delta\in[1,2]}
 \bigg| 	\frac{ \mathbf P\big( Z_s\in x+\varDelta_\delta \big)}{s\cdot \nu\big(x+\varDelta_\delta\big)} -1 \bigg|  \to 0,
 \eeqnn
 which induces that for some $t_0>0$,
 \beqnn
 \sup_{t\geq t_0}  \sup_{0<s\leq t} \sup_{x\geq \theta t} \sup_{y<-b(x)} \sup_{\delta\in[1,2]}
 \frac{ \mathbf P\big( Z_s\in x-y+\varDelta_\delta \big)}{s\cdot \nu\big(x-y+\varDelta_\delta\big)} \leq 2. 
 \eeqnn
 Combining this estimate with \eqref{eqn.312}, together with
 \eqref{eqn.301} for $-b(x)\leq y\leq Kb(x)$, yields
 \eqref{eqn.3011}. 
 \qed

 \bibliographystyle{plain}
 
 \bibliography{Reference}

 \end{document}